\documentclass[a4paper,10pt]{article}
\title{Families of Bianchi modular symbols: critical base-change $p$-adic $L$-functions and $p$-adic Artin formalism}
\author{Daniel Barrera Salazar and Chris Williams\\ \normalsize{ \emph{with an appendix by Carl Wang-Erickson}}}

\date{\vspace{-10pt}}

\newcommand{\Addresses}{{
  \bigskip
  \footnotesize
Daniel Barrera Salazar; Universidad de Santiago de Chile, Dpto. de Matem\'{a}tica y CC., Alameda 3363, Estaci\'{o}n Central, Santiago, Chile $\cdot $  \texttt{daniel.barrera.s@usach.cl}\\

Carl Wang-Erickson; University of Pittsburgh, Thackeray Hall, 166 Thackeray Ave, Pittsburgh, PA 15213, United States $\cdot $ 
 \texttt{carl.wang-erickson@pitt.edu }\\

Chris Williams; University of Warwick, Zeeman Building, Coventry CV4 7AL, United Kingdom\\ $ \cdot $ 
\texttt{christopher.d.williams@warwick.ac.uk}\\

}}
\usepackage{amsmath}

\usepackage{enumitem}

\let\OLDthebibliography\thebibliography
\renewcommand\thebibliography[1]{
\OLDthebibliography{#1}
\setlength{\parskip}{0pt}
\setlength{\itemsep}{0pt plus 0.3ex}
}

\usepackage[margin=1.15in]{geometry}
\usepackage{preamble}

\newcommand{\GL}{\mathrm{GL}}
\newcommand{\cA}{\mathcal{A}}
\DeclareMathOperator{\ad}{ad}

\newcommand{\unitsize}{\#\roi_K^\times}

\newcommand{\FQ}{_{K/\Q}}

\newcommand{\W}{\mathcal{W}}

\newcommand{\E}{\mathcal{E}}
\newcommand{\cE}{\mathcal{E}}
\newcommand{\CC}{\mathcal{C}}

\newcommand{\cyc}{{\mathrm{cyc}}}

\newcommand{\Dz}{\mathcal{D}^0}
\newcommand{\Dl}{\mathscr{D}^0}
\newcommand{\Dla}{\mathcal{D}}
\newcommand{\DD}{\mathscr{D}}
\newcommand{\VV}{\mathscr{V}}

\newcommand{\rigidA}{\mathcal{A}^0}
\newcommand{\laA}{\mathcal{A}}

\newcommand{\hc}{\h^1_{\mathrm{c}}}
\newcommand{\ssh}{^{\leq h}}
\newcommand{\Y}{Y_1(\n)}

\newcommand{\locx}{_{\m_x}}
\newcommand{\locl}{_{\m_\lambda}}

\DeclareMathOperator{\Tor}{Tor}

\DeclareMathOperator{\Sp}{Sp}

\newcommand{\mel}{\mathrm{Mel}}

\newcommand{\Epar}{\E_{\mathrm{par}}}
\newcommand{\Refrm}{\mathrm{ref}}
\newcommand{\bc}{\mathrm{bc}}

\numberwithin{equation}{section}

\begin{document}
\setlength{\abovedisplayskip}{7pt}
	\setlength{\belowdisplayskip}{7pt}
	\setlength{\abovedisplayshortskip}{7pt}
	\setlength{\belowdisplayshortskip}{7pt}

%
%

\maketitle

	\renewcommand{\thefootnote}{\fnsymbol{footnote}} 

\begin{abstract}
Let $K$ be an imaginary quadratic field. In this article, we study the eigenvariety for $\GLt/K$, proving an \'etaleness result for the weight map at non-critical classical points and a smoothness result at base-change classical points. We give three main applications of this; let $f$ be a $p$-stabilised newform of weight $k \geq 2$ without CM by $K$.  Suppose $f$ has finite slope at $p$ and its base-change $f_{/K}$ to $K$ is $p$-regular. Then: (1)  We construct a two-variable $p$-adic $L$-function attached to $f_{/K}$ under assumptions on $f$ that conjecturally always hold, in particular with no non-critical assumption on $f/K$. (2) We construct three-variable $p$-adic $L$-functions over the eigenvariety interpolating the $p$-adic $L$-functions of classical base-change Bianchi cusp forms. (3) We prove that these base-change $p$-adic $L$-functions satisfy a $p$-adic Artin formalism result, that is, they factorise in the same way as the classical $L$-function under Artin formalism.

\end{abstract}

\setcounter{tocdepth}{2}
\footnotesize
\tableofcontents
\normalsize

\section{Introduction}

Let $f \in S_{k+2}(\Gamma_1(N))$ be a classical eigenform of weight $k+2\geq 2$ and level $N$ divisible by $p$ (which we may always assume after possibly $p$-stabilising from prime-to-$p$ level). Suppose $U_p f \neq 0$, i.e.\ $f$ has finite slope. When considering questions about the Iwasawa theory of $f$, there are two notions that occur repeatedly; firstly, base-changing $f$ to an imaginary quadratic field $K$, studying $p$-adic $L$-functions attached to the resulting Bianchi modular form (e.g. \cite{BD07}, \cite{SU14}),  and secondly, allowing $f$ (and the associated $p$-adic $L$-functions) to vary in a $p$-adic family. 

Previous constructions of $p$-adic $L$-functions attached to Bianchi modular forms have focused exclusively on the case of `non-critical slope', namely, under the hypothesis that the slope -- the $p$-adic valuation of the $U_p$-eigenvalue $\alpha_p(f)$ -- is `sufficiently small'. This misses many interesting cases (see Rem.\ \ref{rem:missing cases}). For example, if $p$ is split in $K$ and $E/\Q$ is an elliptic curve with good ordinary reduction at $p$, then its corresponding modular form $f_E$ has level $M$ prime to $p$ and two $p$-stabilisations $f, f'$ to level $N = Mp$, one of which has critical slope base-change to $K$. If $p$ is inert or ramified, the situation is even more pronounced, as fully half of the possible range of slopes of base-change forms is critical. This gives a `missing' base-change $p$-adic $L$-function for \emph{every} modular form of prime-to-$p$ level, and some examples with \emph{no} existing base-change $p$-adic $L$-functions. Further, variation of these $p$-adic $L$-functions in families has previously been proved only for $p$ split, and little is known beyond the case of Hida families, i.e.\ for slope exactly 0.

\subsection{Our results} In this paper, we aim towards a more complete theory of $p$-adic $L$-functions attached to base-change Bianchi modular forms.   Suppose:
	\begin{itemize}\setlength{\itemsep}{0pt}
		\item[(a)] $f$ does not have CM by $K$, 
		\item[(b)] $f$ is either new or a $p$-stabilisation of a newform of level prime to $p$, and 
		\item[(c)] its base-change $f_{/K}$ to $K$ is $p$-regular (in the sense of (C3) in Def.\ \ref{def:pstab newform}).
	\end{itemize}
	
Our main result, Thm.\ \ref{intro:interpolation}, is the variation of $p$-adic $L$-functions in a $p$-adic family through $f_{/K}$, via the construction of a three-variable $p$-adic $L$-function. Our construction is unconditional when $f_{/K}$ is non-critical, and more generally is valid under hypotheses on $f_{/K}$ that conjecturally always hold. This generalises previous constructions for Hida families when $p$ is split in $K$, and when $p$ is inert or ramified, we believe this construction to be entirely new, even for Hida families.
	
We apply this in Thm.\ \ref{intro:thm B} to construct the `missing' $p$-adic $L$-functions attached to $p$-regular base-change forms. Suppose $f$ as above has critical slope. Under the same (conjecturally automatic) hypotheses, we construct a $p$-adic $L$-function attached to $f_{/K}$. We make no non-criticality assumptions and allow arbitrary $p$. This $p$-adic $L$-function naturally has two (cyclotomic and anticyclotomic) variables; note that it is difficult to construct this directly from $p$-adic $L$-functions attached to $f$, as we would not see any anticyclotomic variation.

After restricting to the cyclotomic line, in Thm.\ \ref{intro:artin} we relate our constructions to the $p$-adic $L$-functions attached to the original form $f$ via a $p$-adic analogue of the factorisation given by classical Artin formalism. This generalises a result that is important in the Iwasawa theory of elliptic curves, and that was previously known only under an even more restrictive slope condition that excludes cases of important arithmetic interest.

To prove these results, we build on work of the second author, who gave constructions of $p$-adic $L$-functions for (non-critical slope) Bianchi modular forms, under no base-change assumption, in \cite{Wil17}. The $p$-adic $L$-function was shown to be the Mellin transform of a class in overconvergent cohomology/modular symbols, as introduced by Stevens in the classical setting \cite{Ste94}. This class, and hence the $p$-adic $L$-function, is canonical up to $p$-adic scalar. The main input of the current paper is a pairing of this construction with a systematic study of the eigenvariety parametrising Bianchi modular forms. The use of overconvergent cohomology in constructing eigenvarieties -- generalising the pioneering work of Hida in the ordinary setting -- was known to Stevens, later explored by Ash--Stevens \cite{AS08}, Urban \cite{Urb11} and more recently by Hansen \cite{Han17} and the authors \cite{BW20}. 

\subsection{The Bianchi eigenvariety}  The Bianchi situation is very different from the classical and Hilbert settings, and  requires new ideas. Most strikingly, cuspidal Bianchi modular forms contribute to the cohomology of the associated locally symmetric space in more than one degree (in degrees 1 and 2). The $p$-adic $L$-functions of \cite{Wil17} are constructed using classes in $\hc$; but Bianchi eigenvarieties are naturally constructed using classes that appear \emph{only} in $\h^2_{\mathrm{c}}$ (Lem.\ \ref{lem:min degree}). In particular, the use of existing techniques to vary classes in $\hc$ in families is obstructed by the existence of classical classes in $\h_{\mathrm{c}}^2$. This is related to the fact that $\GL_2(\C)$ does not admit discrete series, and the classical points in the Bianchi eigenvariety -- the points corresponding to classical Bianchi modular forms -- are not Zariski-dense. This is captured in the phenomenon that the cuspidal part of the Bianchi eigenvariety is one-dimensional lying over a two-dimensional weight space.

Overcoming this obstruction is a key step in the construction, and occupies all of \S\ref{sec:families of ms}. We isolate certain curves $\Sigma$ in weight space that allow us to pass from families in $\h^2_{\mathrm{c}}$ to families in $\hc$ (Prop.\ \ref{prop:appear in H1}), as needed to relate to $p$-adic $L$-functions. Under a smoothness hypothesis on $\Sigma$, that is satisfied for all base-change families, we prove an \'etaleness result (Thm.\ \ref{thm:free rank 1}) for the eigenvariety at non-critical classical points. These results require no base-change condition. Again, the situation is made difficult by existence of classes in $\h^2_{\mathrm{c}}$; we hope the techniques we use will apply in more general `badly behaved' settings, e.g.\ $\GL_n/\Q$ for $n\geq 3$.

In \S\ref{sec:parallel weight ev}  we construct a `parallel weight' eigenvariety $\Epar$ using degree 1 overconvergent cohomology over the parallel weight line in the Bianchi weight space. In $\Epar$, we recover the desirable property that the classical points are Zariski-dense, and by $p$-adic Langlands functoriality it contains all classical points corresponding to base-change forms. In Props.\ \ref{prop:smooth nc} and \ref{prop:smooth} we prove that the base-change eigenvariety $\cE_{\mathrm{bc}} \subset \Epar$ is smooth and reduced at decent (see Def.\ \ref{def:decent}) classical base-change points $f_{/K}$ (under no non-criticality assumption). The result uses a base-change deformation functor described in an appendix by Carl Wang-Erickson that allows us to reduce to a case treated by Bella\"{i}che \cite{Bel12}.

We say $f_{/K}$ is \emph{$\Sigma$-smooth} if the inclusion $\E_{\mathrm{bc}} \subset \Epar$ is locally an isomorphism at $f_{/K}$. For decent $f_{/K}$, this is equivalent to $\Epar$ being smooth at $f_{/K}$. By the \'etaleness result of Thm.\ 4.5, any non-critical $f_{/K}$ is $\Sigma$-smooth. More generally, a conjecture of Calegari and Mazur \cite{CM09} -- which predicts that the only classical families of Bianchi modular forms are CM or come from base-change -- would imply that every $f_{/K}$ is $\Sigma$-smooth. We discuss this in detail in \S\ref{sec:sigma smooth}.

\subsection{Applications to $p$-adic $L$-functions} We summarise our main applications.  Let $f \in S_{k+2}(\Gamma_1(N))$ satisfy (a), (b) and (c) above, and suppose $f$ is decent and $f_{/K}$ is $\Sigma$-smooth. Let $\mathscr{X}(\cl_K(p^\infty))$ be the two-dimensional rigid space of $p$-adic characters on $\cl_K(p^\infty)$, the ray class group of $K$ of conductor $p^\infty$. Let $\phi$ be a finite order Hecke character of $K$ of conductor prime to $p\roi_K$. Let $x_f$ be the point in the Coleman--Mazur eigencurve $\CC$ corresponding to $f$, and let $V_{\Q}$ be a neighbourhood of $x_f$ in $\CC$.  If $y \in V_{\Q}$ is a classical point, write $f_y$ for the corresponding classical eigenform. For a Zariski-dense set of classical $y \in V_{\Q}$, the base-change $f_{y/K}$ is non-critical (Def.\ \ref{def:non-critical}). In \S\ref{sec:subsec three variable}, in the language of distributions, we prove:

\begin{theorem-intro}\label{intro:interpolation} 
	Up to shrinking $V_{\Q}$, and for sufficiently large $L \subset \overline{\Q}_p$, there exists a unique rigid-analytic function 
	\[
		\mathcal{L}_p^\phi : V_{\Q} \times \mathscr{X}(\cl_K(p^\infty)) \longrightarrow L
	\]
	such that at any classical point $y \in V_{\Q}(L)$ of weight $k_y + 2$ with non-critical base-change, and any Hecke character $\varphi$ of $K$ with conductor $\ff|p^\infty$ and infinity type $(0,0) \leq (q,r) \leq (k_y,k_y)$, we have
	\begin{equation}\label{eqn:interpolation intro}
		\mathcal{L}_p^\phi(y,\varphi_{p-\mathrm{fin}}) = c_y\left(\textstyle\prod_{\pri|p}Z_{\pri}(\varphi)\right)A(f_{y/K},\varphi)\cdot\Lambda(f_{y/K},\varphi\phi)
	\end{equation}
	where $c_y \in L^\times$ is a $p$-adic period at $y$, $\varphi_{p-\mathrm{fin}} \in \mathscr{X}(\cl_K(p^\infty))$ is the $p$-adic avatar of $\varphi$, $Z_{\pri}(\varphi)$ is an Euler-type factor, $A(f_{y/K},\varphi)$ is an explicit non-zero scalar, and $\Lambda(f_{y/K},-)$ is the $L$-function of $f_{y/K}$, all of which are defined in \S\ref{sec:bms}. For a fixed set $\{c_y\}$ of $p$-adic periods, this $\mathcal{L}_p^\phi$ is unique.
\end{theorem-intro}

In particular, for $y \in V_{\Q}$ as in the theorem, the specialisation of $\mathcal{L}_p^\phi$ at $y$ is precisely the two-variable $p$-adic $L$-function (twisted by $\phi)$ attached to $f_{y/K}$. Thus $\mathcal{L}_p^\phi$ is the three-variable $p$-adic $L$-function attached to $V_{\Q}$ and $\phi$. The meat of the proof is in producing a canonical class (up to scaling) in the overconvergent cohomology over the eigenvariety, interpolating the overconvergent classes of \cite{Wil17} at classical points. We do this in \S\ref{sec:families of ms}-\ref{sec:three variable} using the local geometry of the parallel weight eigenvariety. Given this class, $\mathcal{L}_p^\phi$ is defined as its (twisted) Mellin transform.

We comment briefly on the `choice' of periods $c_y$. We do have some control over them; under a non-vanishing hypothesis, which is satisfied for $f_{/K}$ non-critical and conjecturally for all $f_{/K}$, any two systems of periods $\{c_y\}, \{c_y'\}$ for which there exists a three-variable $p$-adic $L$-function are of the form $c_y = \alpha(y)c_y'$, where $\alpha \in \roi(V_{\Q})^\times$. We show this in Prop.\ \ref{prop:well-defined}.

When $p$ splits in $K$, using the strategy and results developed in the present paper, in \cite{BetWil20} Thm.\ \ref{intro:interpolation} is proved in the $p$-irregular case (i.e. when (c) fails).

	\begin{remark*}
It is natural to ask if there are analogues of Thm.\ \ref{intro:interpolation} for Bianchi modular forms that are \emph{not} base-change. A general Bianchi modular form $\f$ still varies in a 1-dimensional $p$-adic family $V$ of overconvergent Bianchi modular symbols (Thm.\ \ref{thm:dimension 1}) over a curve $\Sigma$ in the (2-dimensional) Bianchi weight space. When $\f$ is non-critical and $\Sigma$ is smooth at the weight of $\f$, then we prove $V$ is \'etale over $\Sigma$ (Thm.\ \ref{thm:free rank 1}, Cor.\ \ref{prop:free neighbourhood}). In this case we prove existence of a rigid function $\mathcal{L}_p^\phi : V \times \mathscr{X}(\cl_K(p^\infty)) \to L$ satisfying the interpolation \eqref{eqn:interpolation intro} of Thm.\ \ref{intro:interpolation}. 

If $V$ is a classical family (i.e.\ the classical points in $V$ are Zariski-dense), then $\Sigma$ is parallel, hence smooth. We thus unconditionally construct three-variable $p$-adic $L$-functions around non-critical points in \emph{every} classical family.

In general the nature of $V$ is mysterious; it might contain only finitely many classical points, so \eqref{eqn:interpolation intro} could be empty outside $\f$! We give a possible arithmetic interpretation in this case, and the construction of $\mathcal{L}_p^\phi$ in general, in Rem.\ \ref{rem:non-base-change}.
\end{remark*}

Now suppose $f_{/K}$ is critical and $\Sigma$-smooth, so we cannot use \cite{Wil17}. Using Thm.\ \ref{intro:interpolation}, we define the `missing' $p$-adic $L$-function for $f_{/K}$ to be the specialisation $L_p^\phi(f_{/K},-) \defeq \mathcal{L}_p^\phi(x_f,-)$. This is a two-variable $p$-adic $L$-function attached to $f_{/K}$ satisfying the expected growth property and which we prove is canonical up to a $p$-adic scalar (corresponding to the $p$-adic period of $f$). In Thm.\ \ref{prop:critical}, we prove:

\begin{theorem-intro}\label{intro:thm B}
	Suppose $f_{/K}$ is critical and $\Sigma$-smooth. Then we have 
			\[
				L_p^\phi(f_{/K},\varphi_{p-\mathrm{fin}}) = 0
			\]
			for all $\varphi$ of conductor $\ff|p^\infty$ and infinity type $(0,0)\leq(q,r)\leq (k,k)$.
\end{theorem-intro}

	 Thm.\ \ref{intro:thm B} gives at least a conjectural construction for every $p$-regular finite slope $f_{/K}$, so it goes much further than \cite{Wil17} (see Rem.\ \ref{rem:missing cases}).

The $L$-functions of $f$ and $f_{/K}$ are related by Artin formalism $L(f_{/K},s) = L(f,s) \cdot L(f,\chi_{K/\Q},s)$, where $\chi_{K/\Q}$ is the quadratic character associated to $K/\Q$. For the $p$-adic $L$-functions, such a factorisation does not make sense on the nose, since $L_p(f_{/K})$ is two-variabled whilst $L_p(f)$ and $L_p^{\chi_{K/\Q}}(f)$ are both one-variabled (valued on $\mathscr{X}(\cl_{\Q}^+(p^\infty)) \cong \mathscr{X}(\Zp^\times)$). We fix this by letting $L_p^{\cyc}(f_{/K})$ denote the restriction of $L_p(f_{/K})$ to the cyclotomic line. In Thm.\ \ref{thm:factorisation}, we then prove:

\begin{theorem-intro}\label{intro:artin}
	Suppose $f_{/K}$ is $\Sigma$-smooth and that $L_p^\cyc(f_{/K})$ and $L_p(f)L_p^{\chi_{K/\Q}}(f)$ are both non-zero. Then $L_p^{\mathrm{cyc}}(f_{/K}) = L_p(f)L_p^{\chi_{K/\Q}}(f)$ as distributions on $\Zp^\times$.
\end{theorem-intro}

This is a $\GL_2$-analogue of Gross's ($\GL_1$) $p$-adic Artin formalism relating Katz and Kubota--Leopoldt $p$-adic $L$-functions \cite{Gro80}. A factorisation relating Rankin and symmetric square $p$-adic $L$-functions, again mimicking classical Artin formalism, has been obtained by Dasgupta \cite{Das16}.

We remark that since our periods are only defined up to an algebraic scalar, this is really an equality of one-dimensional lines in the (infinite-dimensional) space $\roi(\mathscr{X}(\cl_{\Q}^+(p^\infty)))$. We explain this more fully in \S\ref{sec:base-change evs}. When $f$ has sufficiently small slope -- namely, slope $h < (k+1)/2$ -- this theorem is automatic from \emph{classical} Artin formalism, since both sides satisfy a growth property that renders this line unique with respect to their interpolation properties. For more general slopes, this result is far from obvious, as the growth and interpolation properties are satisfied by an infinite number of distinct lines in $\roi(\mathscr{X}(\cl_{\Q}^+(p^\infty)))$, and we really require the additional input of our three-variable $p$-adic $L$-function to see the equality. 

Ideally, we would also be able to control the ($p$-adic and archimedean) periods integrally to pin down an equality of lattices within this line, but this seems an extremely subtle question; studying integral period relations in the Bianchi base-change case was already the subject of \cite{TU18}, even before studying this in the context of $p$-adic $L$-functions. We comment further in \S\ref{factorisation}.

The non-vanishing condition is automatically satisfied when $f$ and $f_{/K}$ are non-critical. Under a conjecture of Greenberg, which says that all critical elliptic modular forms are CM,  we expect that $L_p(f)$ and $L_p^{\chi_{K/\Q}}(f)$ can be related to Katz $p$-adic $L$-functions, and are always non-zero; work in this direction is explained in \cite{BelCM}. We conjecture that $L_p^\cyc(f_{/K})$ is similarly never zero.

A case of particular interest where this theorem applies is the following. Let $E/\Q$ be an elliptic curve with good supersingular reduction at odd $p$, and let $f_{\alpha}$ be a $p$-stabilisation of the corresponding weight 2 classical modular form corresponding to a root $\alpha$ of the Hecke polynomial at $p$. Since $h = v_p(\alpha) = 1/2$ and $k=0$, this is outside the range where Thm.\ \ref{intro:artin} is automatic. Suppose $p$ splits in $K$. Then the base-change $f_{\alpha/K}$ has (non-critical) slope $1/2$ at each of the primes above $p$. Since the $L$-function of $f_{\alpha/K}$ corresponds to a $p$-depleted $L$-function for $E/K$, we get a factorisation
\[
L_{p,\alpha}^{\cyc}(E/K) = L_{p,\alpha}(E/\Q)L_{p,\alpha}(E^{\chi_{K/\Q}}/\Q)
\]
of the $p$-adic $L$-function of $E/K$ in terms of the $p$-adic $L$-functions of $E$ and its quadratic twist by $\chi_{K/\Q}$. In the ordinary case this factorisation was required in Skinner and Urban's proof of the Iwasawa main conjecture (see \cite{SU14}).

Finally, we remark that modulo the existence of anticyclotomic $p$-adic $L$-functions in Coleman families, the same methods also apply to restriction to the anticyclotomic line. In this case, under the same non-vanishing hypothesis, we obtain $L_p^{\mathrm{anti}}(f_{/K}) = L_p^{\mathrm{anti}}(f)^2$ (where the anticyclotomic $p$-adic $L$-function exists). We leave the details to the interested reader. Note that anticyclotomic $p$-adic $L$-functions do not yet exist in the case where $f$ is critical. The above suggests that a good candidate for (the square of) an anticyclotomic $p$-adic $L$-function in this case is the restriction to the anticyclotomic line of the $p$-adic $L$-function attached to $f_{/K}$ in this paper. \\[5pt]
\textbf{Acknowledgements:} 
This paper owes a great debt to David Hansen, who allowed us to reproduce his unpublished results, and David Loeffler, who suggested we work on $p$-adic Artin formalism. Both also gave invaluable feedback on earlier versions of this paper. We would particularly like to thank Adel Betina, for his contribution to the proof of Lem.\ \ref{prop:gen by 1}, James Newton, who patiently answered our questions on his work, and Carl Wang-Erickson, who provided the appendix. The project also benefited from interesting and motivating discussions with Chris Birkbeck, Denis Benois, Kevin Buzzard, Mladen Dimitrov, Netan Dogra, Luis Palacios, Lennart Gehrmann, Ming-Lun Hsieh, Adrian Iovita, Fabian Januszewski and Victor Rotger.  Finally, we thank the referee for their valuable comments and corrections. C.W.\ would like to thank Victor Rotger for his generous financial support, and was also supported by EPSRC postdoctoral fellowship EP/T001615/1. D.B.\ was supported by ANID's grants 77180007 and 11201025. This project has received funding from the European Research Council (ERC) under the European Union's Horizon 2020 research and innovation programme (grant agreement No 682152).

%
%

\section{Bianchi modular forms and $p$-adic $L$-functions}
\label{sec: 2}\label{basic definitions}
In this section, we fix notation and recap the results of \cite{Wil17}, always using the conventions \emph{op.\ cit.} Fix embeddings $\overline{\Q} \hookrightarrow \C$ and $\overline{\Q} \hookrightarrow \overline{\Q}_\ell$ for each prime $\ell$. Let $K$ be an imaginary quadratic field with ring of integers $\roi_K$ and discriminant $-d$. Let $p$ be a rational prime; the choice $\overline{\Q} \hookrightarrow \overline{\Q}_p$ fixes a $p$-adic valuation $v_p$ on $\overline{\Q}$.

Denote the adele ring of $K$ by $\A_K = \C \times \A_K^f$, where $\A_K^f$ denotes the finite adeles.  For an ideal $\ff \subset \roi_K$, let $\mathrm{Cl}_K(\ff) \defeq K^\times \backslash \A_K^\times/I(\ff)\C^\times$ denote the ray class group of $K$ modulo $\ff$, where $I(\ff) \defeq \{x \in (\roi_K\otimes \widehat{\Z})^\times: x \equiv 1 \newmod{\ff}\}$. Throughout, we work at level $\n \subset \roi_K$ divisible by each prime of $K$ above $p$. We write $U_1(\n)$ for the open compact subgroup of $\GLt(\roi_K\otimes_{\Z}\widehat{\Z})$ of matrices congruent to $\smallmatrd{*}{*}{0}{1}$ modulo $\n$, and $\mathcal{K}_\infty = \mathrm{SU}_2(\C)\C^\times$, and we define the associated locally symmetric space by 
\[
	\Y \defeq \GLt(K)\backslash\GLt(\A_K)/\mathcal{K}_\infty U_1(\n).
\]
We define $\uhs \defeq \C\times \R_{>0}$; the space $\Y$ decomposes as a finite disjoint union of quotients of $\uhs$. 

Let $j \geq 0$ be an integer. For any ring $R$, let $V_{j}(R)$ denote the space of polynomials over $R$ of degree at most $j$. Throughout, we will denote modules of locally analytic distributions by $\Dla(*)$. These carry an action of $U_1(\n)$, and the corresponding local systems on $\Y$ will be denoted by $\DD(*)$.

We will use $f$ for a classical modular form and $\f$ a Bianchi modular form. For $X$ an affinoid in a rigid space, $\roi(X)$ will denote the ring of rigid functions on $V$. We will write $V$ (resp.\ $V_{\Q}$) for affinoids in the Bianchi (resp.\ Coleman--Mazur) eigenvariety. If $y$ is a classical point in an eigenvariety, we will write $\f_y$ or $f_y$ for the corresponding (Bianchi or classical) normalised modular form of minimal level, which will always be uniquely defined by our running assumptions.

\subsection{Bianchi modular forms, $L$-functions and cohomology} \label{BMF}\label{sect:modsymb}
Let $\lambda = (\mathbf{k},\mathbf{v})$ be a weight, where $\mathbf{k} = (k_1,k_2)$ and $\mathbf{v} = (v_1,v_2)$ are two elements of $\Z^2$. There is a finite-dimensional $\C$-vector space $S_{\lambda}(U_1(\n))$ of \emph{Bianchi cusp forms} of weight $\lambda$ and level $U_1(\n)$, which are vector-valued functions on $\GL_2(\A_K)$ satisfying suitable transformation, harmonicity and growth conditions. These objects are defined precisely in e.g.\ \cite[Def.\ 1.2]{Wil17}. If $k_1 \neq k_2$, then $S_{\lambda}(U_1(\n)) = 0$ (see \cite{Har87}), so we will restrict to parallel weight $k_1 = k_2 = k \geq 0$; and in this case, we can always twist the central character by a power of the norm to assume that $v_1 = v_2 =0$ as well. For the rest of this section, we fix $\lambda = [(k,k),(0,0)]$, and we will write this as $\lambda = (k,k)$ without further comment.

For $\mathfrak{q} \subset \roi_F$ prime, fix a uniformiser $\varpi_{\mathfrak{q}}$ of $K_{\mathfrak{q}}$. Consider $\varpi_{\mathfrak{q}} \in \A_K^{f,\times}$ trivial at every place $ \neq \mathfrak{q}$. We have Hecke operators $T_{\mathfrak{q}} \defeq [U_1(\n) \smallmatrd{1}{}{}{\varpi_{\mathfrak{q}}} U_1(\n)]$. If $\mathfrak{q}|\n$, we write $U_{\mathfrak{q}}$ instead of $T_{\mathfrak{q}}$. For each $v \in \cl_K(\n)$, let $u_v \in \A_{K}^{f,\times}$ be a representative, and define $\langle v \rangle \defeq [U_1(\n)\smallmatrd{u_v}{}{}{u_v}U_1(\n)]$. The double coset operators $T_{\mathfrak{q}}, U_{\mathfrak{q}}, \langle v\rangle$ are all independent of choices of representatives and act on $S_{\lambda}(U_1(\n))$. An \emph{eigenform} is a simultaneous eigenvector. Attached to an eigenform $\f$ is a character $\epsilon_{\f} : \cl_K(\n) \to \overline{\Q}^\times$ with $\langle v \rangle \f = \epsilon_{\f}(v)\f$ for all $v \in \cl_K(\n)$.
	
	\begin{definition}\label{def:hecke alg}
		Let $\HH_{\n,p}$ denote the $\Zp$-algebra generated by the Hecke operators $\{T_{\mathfrak{q}}: (\mathfrak{q},\n) = 1\}, \{U_{\pri}: \pri|p\}$ and $\{\langle v \rangle : v  \in \cl_K(\n)\}$.
	\end{definition}

Our theorems require $\n$ to be divisible by each prime $\pri$ above $p$. If $\pri\nmid \mathfrak{N}$ and $\f \in S_\lambda(U_1(\mathfrak{N}))$ is an eigenform, let $a_{\pri}(\f)$ denote the $T_{\pri}$ eigenvalue of $\f$, and let $\alpha_{\pri}$ and $\beta_{\pri}$ denote the roots of the Hecke polynomial $X^2 - a_{\pri}(\f)X + \epsilon_{\f}(\pri)N(\pri)^{k+1}.$ The \emph{$\pri$-stabilisations of $\f$} are 
\[
	\f_{\alpha_{\pri}}(g) \defeq \f(g) - \beta_{\pri}\f\left(\smallmatrd{\varpi_{\pri}^{-1}}{0}{0}{1}g\right), \hspace{8pt}\f_{\beta_{\pri}}(g) \defeq \f(g) - \alpha_{\pri}\f\left(\smallmatrd{\varpi_{\pri}^{-1}}{0}{0}{1}g\right).
\]
Then $\f_{\alpha_{\pri}}$ and $\f_{\beta_{\pri}}$ are eigenforms of level $U_1(\mathfrak{N}\pri)$ with $U_{\pri}$-eigenvalues $\alpha_{\pri}$ and $\beta_{\pri}$.

Throughout the paper, we work with the following Bianchi modular forms:

\begin{customconditions}{2.2[$K$]}\label{def:pstab newform}\label{running assumptions}
	Let $\lambda = (k,k)$ and $\n \subset \cO_F$ divisible by each $\pri|p$. Let $\f \in S_\lambda(U_1(\n))$ be a finite slope $p$-regular $p$-stabilised newform, in the sense that:
	\begin{itemize}\setlength{\itemsep}{0pt}
		\item[(C1)]  $\f$ is an eigenform, and for each $\pri|p$, we have $U_{\pri}\f = \alpha_{\pri}\f$ with $\alpha_{\pri} \neq 0$;
		\item[(C2)]  there exist $S \subset \{\pri|p\}$, $\mathfrak{N}$ prime to $S$, and a newform $\f_{\mathrm{new}} \in S_{\lambda}(U_1(\mathfrak{N}))$ such that $\n = \mathfrak{N}\textstyle\prod_{\pri \in S}\pri$ and $\f$ is obtained from $\f_{\mathrm{new}}$ by $\pri$-stabilising for $\pri \in S$;
		\item[(C3)]  for each $\pri \in S$, $X^2 - a_{\pri}(\f_{\mathrm{new}})X + \epsilon_{\f_{\mathrm{new}}}(\pri)N(\pri)^{k+1}$ has distinct roots.
	\end{itemize}
	Note newforms of level $\n$ themselves satisfy (C2),(C3) with $S = \varnothing$.
\end{customconditions}
\stepcounter{definition}

Let $\f$ satisfy Conditions \ref{running assumptions} and let $\Lambda(\f,\varphi)$ denote the (completed) $L$-function of $\f$, normalised as in \cite{Wil17}. Here $\varphi$ runs over Hecke characters of $K$. By \cite[Thm.\ 8.1]{Hid94}, we see that there exists a period $\Omega_{\f} \in \C^\times$ and a number field $E$ containing the Hecke eigenvalues of $\f$ such that if $\varphi$ is an algebraic Hecke character of infinity type $0 \leq (q,r) \leq (k,k)$ with $q,r \in \Z$, we have
\begin{equation}\label{period}
\Lambda(\f,\varphi)/\Omega_{\f} \in E(\varphi),
\end{equation}
where $E(\varphi) \subset \overline{\Q}$ is the extension of $E$ generated by the values of $\varphi$.

\subsection{Base-change}

	Let $f_{\mathrm{new}} \in S_{k+2}(\Gamma_1(N))$ be a classical cuspidal newform of nebentypus $\epsilon_{f_{\mathrm{new}}}$, generating an automorphic representation $\pi$ of $\GL_2(\A_{\Q})$. Let $\mathrm{BC}(\pi)$ be the base-change of $\pi$ to $\GLt(\A_K)$ (see \cite{Lan80}). The \emph{base-change of $f$ to $K$} is the normalised new vector $\f_{\mathrm{new}}$ in $\mathrm{BC}(\pi)$, which is a Bianchi modular form of weight $(k,k)$. If $f_{\mathrm{new}}$ has level $N$, the level of $\f_{\mathrm{new}}$ is an ideal $\n \subset \roi_K$ with $\tfrac{N}{(N,d)}\roi_K|\n|N\roi_K$, recalling $-d = \mathrm{disc}(K)$ (see \cite[\S2.1]{Fri83}); so if $(N,d) = 1$, then $\n = N\roi_K$.  If $f_{\mathrm{new}}$ does not have CM by $K$, then $\f_{\mathrm{new}}$ is cuspidal.

	If $p \nmid N$, let $\alpha_p,\beta_p$ be the roots of $X^2 - a_p(f_{\mathrm{new}})X + \epsilon_{f_{\mathrm{new}}}(p)p^{k+1}$, and for $\pri|p$, let $\alpha_{\pri},\beta_{\pri}$ be the roots of $X^2 - a_{\pri}(\f_{\mathrm{new}})X + \epsilon_{\f_{\mathrm{new}}}(\pri)N(\pri)^{k+1}$. If $p$ is split or ramified in $K$, then we can take $\alpha_{\pri} = \alpha_p, \beta_{\pri} = \beta_p$; and if $p$ is inert, then we may take $\alpha_{\pri} = \alpha_p^2, \beta_{\pri} = \beta_p^2$. If $f_\alpha$ (resp.\ $f_\beta$) is the $p$-stabilisation of $f_{\mathrm{new}}$ corresponding to $\alpha_p$ (resp.\ $\beta_p$), we define its base-change to be the $p$-stabilisation $\f_{\alpha\alpha}$ (resp.\ $\f_{\beta\beta})$ of $\f_{\mathrm{new}}$ corresponding to $\alpha_{\pri}$ (resp.\ $\beta_{\pri}$) for all $\pri|p$.

	We will consider the following classical analogue of Conditions \ref{running assumptions}:
	
	\begin{customconditions}{2.2[$\Q$]}\label{running assumptions Q}
			Let $N$ be divisible by $p$. Let $f \in S_{k+2}(\Gamma_1(N))$ such that:
		\begin{itemize}\setlength{\itemsep}{0pt}
			\item[(C1$'$)] $f$ is an eigenform, and $U_pf = \alpha_{p}f$ with $\alpha_p \neq 0$;
			\item[(C2$'$)]  $f$ is new or the $p$-stabilisation of a newform $f_{\mathrm{new}}$ of level prime to $p$;
			\item[(C3$'$)]  If $f$ is the $p$-stabilisation of $f_{\mathrm{new}}$, then $\alpha_p \neq \beta_p$. If $p$ is inert, $a_p(f_{\mathrm{new}}) \neq 0$;
			\item[(C4$'$)]  $f$ does not have CM by $K$.
		\end{itemize}
	\end{customconditions}

\begin{remark} 
	We explain the extra condition in (C3$'$). We say $f_{\mathrm{new}}$ is $p$-regular if $\alpha_p \neq \beta_p$. Conjecturally, every such $f_{\mathrm{new}}$ is $p$-regular (see e.g.\ \cite[\S1.3]{Bel12}). From the description of $\alpha_{\pri}$ and $\beta_{\pri}$ in terms of $\alpha_p$ and $\beta_p$, we see if $f_{\mathrm{new}}$ is $p$-regular, then its base-change $\f_{\mathrm{new}}$ is $p$-regular \emph{except} if $p$ is inert and $\alpha_p^2 = \beta_p^2$, which occurs if and only if $a_p(f_{\mathrm{new}}) = 0$. Thus if $f$ satisfies Conditions \ref{running assumptions Q}, then its base-change $\f$ (which is cuspidal by (C4$'$)) satisfies Conditions \ref{running assumptions}.
	
	When $p$ splits in $K$, in \cite{BetWil20} Thm.\ \ref{intro:interpolation} is proved for $p$-irregular $\f$.
\end{remark}

\subsection{Classical and overconvergent cohomology}\label{sec:overconvergent cohomology}
\begin{definition}For a ring $R$ and $\lambda = (k,k)$, let $V_\lambda = V_{k,k}(R) \defeq V_k(R)\otimes_{R}V_k(R)$. (We think of $V_\lambda$ as polynomials on $\roi_K\otimes_{\Z}\Zp$ that have degree at most $k$ in each variable). This space has a natural left action of $\GLt(R)^2$ induced by the action of $\GLt(R)$ on each factor by
\[\smallmatrd{a}{b}{c}{d}\cdot P(z) = (a+cz)^k P\left(\textstyle\frac{b+dz}{a+cz}\right),\]
inducing a right action on the dual $V_{\lambda}(R)^* \defeq \mathrm{Hom}(V_{\lambda}(R),R).$ When $R$ is a $K$-algebra, this gives a local system $\VV_{\lambda}(R)^*$ on the space $\Y$ (denoted $\mathcal{L}_1(V_\lambda(R)^*)$ in \cite[Def.\ 4.2]{BW_CJM}). The Hecke algebra $\HH_{\n,p}$ acts on $\hc(\Y,\VV_\lambda(R)^*)$ as usual (e.g. \cite[p.346--7]{Hid88}), and by \cite[\S3,\S8]{Hid94} and \cite{Har87}, we have:
\end{definition}

\begin{theorem}\label{thm:mult one} There is a Hecke-equivariant injection
	\[
		S_{\lambda}(U_1(\n)) \hookrightarrow \hc(\Y,\VV_{\lambda}(\C)^*),\hspace{12pt} \f \mapsto \phi_{\f}.
	\]
Let $\f \in S_{\lambda}(U_1(\n))$ satisfy Conditions \ref{running assumptions}. Then the generalised $\f$-eigenspace $\hc(\Y,\VV_{\lambda}(\C)^*)_{(\f)}$ for $\HH_{\n,p}$ is 1-dimensional, and $\phi_{\f}/\Omega_{\f}$ has coefficients in $\VV_{\lambda}(E)^*,$ for $\Omega_{\f} \in \C^\times$ and $E$ as in (\ref{period}).
\end{theorem}

Let $R$ be an $(\roi_K\otimes_{\Z}\Zp)$-algebra, and $L$ a finite extension of $\Qp$.

\begin{definition}\label{def:A}
Let $\laA(R)$ be the space of locally analytic functions $\roi_K\otimes_{\Z}\Zp \rightarrow R$. When $R = L$, we equip this space with a weight $\lambda$ action of the semigroup
\[
\Sigma_0(p) \defeq \left\{\smallmatrd{a}{b}{c}{d} \in M_2(\roi_K\otimes_{\Z}\Zp): v_{p}(c) > 0 \hspace{3pt}\forall \pri|p, a\in(\roi_K\otimes_{\Z}\Zp)^\times, ad-bc \neq 0\right\}
\]
by setting $\smallmatrd{a}{b}{c}{d}\cdot \zeta (z) = (a+cz)^k\zeta\left(\frac{b+dz}{a+cz}\right).$  Write $\laA_\lambda(L)$ for $\laA(L)$ with this action. As $\n$ is divisible by each prime above $p$, $U_1(\n)$ acts on $\laA_\lambda(L)$ by projection to $p$.
\end{definition}
\begin{definition}
Let $\Dla(R) \defeq \mathrm{Hom}_{\mathrm{cts}}(\laA(R),R)$ be the space of $R$-valued locally analytic distributions on $\roi_K\otimes_{\Z}\Zp$. When $R = L$ as above, we write $\Dla_{\lambda}(L)$ for this space equipped with the weight $\lambda$ right action of $\Sigma_0(p)$ given by $\mu|\gamma(\zeta) = \mu(\gamma\cdot\zeta).$ Then $\Dla_{\lambda}(L)$ gives rise to a local system on $\Y$, which we denote by $\DD_{\lambda}(L)$. In \cite[Def.\ 4.2]{BW_CJM} this local system is denoted $\mathcal{L}_2(\Dla_\lambda(L))$.
\end{definition}

There is a natural map $\Dla_{\lambda}(L) \rightarrow V_{\lambda}(L)^*$ given by dualising the inclusion of $V_{\lambda}(L)$ into $\laA(L)$. For each $i$, this induces a \emph{specialisation map}
\[\rho_\lambda: \h_{\mathrm{c}}^i(\Y,\DD_{\lambda}(L)) \longrightarrow \h_{\mathrm{c}}^i(\Y,\VV_{\lambda}(L)^*).
\]
\begin{definition}\label{def:non-critical}
Let $\f \in S_{\lambda}(U_1(\n))$ be an eigenform. We say $\f$ is \emph{non-critical} if $\rho_\lambda$ becomes an isomorphism (for each $i$) upon restriction to the generalised eigenspaces of the Hecke operators at $\f$. If $\f$ is non-critical, let $\Psi_{\f} \in \hc(\Y,\DD_{\lambda}(L))$ denote the unique lift of $\phi_{\f}/\Omega_{\f}$ considered with $L$-coefficients, where we assume $L$ contains all embeddings of the fields $K$ and $E$ (from \eqref{period}).
\end{definition}

\begin{definition}
If $\f \in S_{\lambda}(U_1(\n))$ is an eigenform, we say $\f$ has \emph{small slope} (or \emph{non-critical slope}) if $v_{p}(\alpha_{\pri})< (k+1)/e_{\pri}$ for all $\pri|p$, where $U_{\pri}\f = \alpha_{\pri}\f$ and $e_{\pri}$ is the ramification degree of $\pri$.
\end{definition}

\begin{theorem}\cite[Thm.\ 8.7]{BW_CJM}\label{control theorem}
If $\f$ has small slope, then $\f$ is non-critical.
\end{theorem}

\subsection{Modular symbols and Mellin transforms}\label{sec:bms}
Let $\Delta_0 \defeq \mathrm{Div}^0(\Proj(K))$ denote the space of `paths between cusps' in $\uhs$, and let $V$ be any right $\Sigma_0(p)$-module. For a discrete subgroup $\Gamma\subset\Sigma_0(p)\cap\SLt(K)$, define the space of \emph{$V$-valued modular symbols for $\Gamma$} to be the space
\[\symb_{\Gamma}(V) \defeq \mathrm{Hom}_\Gamma(\Delta_0,V)\]
of functions satisfying the $\Gamma$-invariance property that $(\phi|\gamma)(\delta)\defeq \phi(\gamma \delta)|\gamma = \phi(\delta)$ for all $\delta\in\Delta_0, \gamma\in\Gamma,$ where $\Gamma$ acts on the cusps by $\smallmatrd{a}{b}{c}{d}\cdot r = (ar+b)/(cr+d).$ 
It also acts naturally on $\uhs$, and by \cite[Prop.\ 8.2]{BW17} we have an isomorphism $\hc(\Gamma\backslash\uhs, \VV) \cong \symb_\Gamma(V),$ where $\VV$ is the corresponding local system on $\Gamma\backslash\uhs$.

The space $\Y$ decomposes as a disjoint union of spaces $\Gamma_i\backslash\uhs$, for $\Gamma_i\subset \GLt(K)$ discrete subgroups indexed by $i \in \cl_K = \cl_K(\roi_K)$ (see \cite[\S4.2.2]{BW_CJM}).  Each $\Gamma_i$ depends on a choice of representative $t_i \in \A_K^{f,\times}$ of $i \in \cl_K$. From the above this induces a (non-canonical) decomposition
\begin{equation}\label{eq:decomposition}
			\hc(\Y,\VV) \cong \textstyle\bigoplus_{i\in\mathrm{Cl}_K}\symb_{\Gamma_i}(V).
	\end{equation}
 When $V$ is a $\Sigma_0(p)$-module, there is a natural action of the Hecke algebra $\HH_{\n,p}$ on the direct sum, defined as in \cite[\S3.3]{Wil17}, and \eqref{eq:decomposition} is Hecke-equivariant.


Let $\mathrm{Cl}_K(p^\infty) \defeq K^\times\backslash\A_K^\times/\C^\times\prod_{v\nmid p}\roi_v^\times$.  We have a decomposition $\cl_K(p^\infty) = \sqcup_{i\in\cl_K} \cl_K^i(p^\infty)$, where $\cl_K^i(p^\infty)$ is the fibre of $i$ under the canonical surjection $\cl_K(p^\infty) \twoheadrightarrow \cl_K$. The choice of representative $t_i \in \A_K^{f,\times}$ identifies $\cl_K^i(p^\infty)$ non-canonically with $(\roi_K\otimes_{\Z}\Zp)^\times/\roi_K^\times$. 
Let $R$ be an $(\roi_K\otimes_{\Z}\Zp)$-algebra such that $\Dla(R)$ carries a right action of $U_1(\n)$, hence giving rise to a local system on $\Y$. 
Let $\Psi \in \hc(\Y,\DD(R))$, and write
$\Psi = (\Psi^1,...,\Psi^h)$ with each $\Psi^i \in \symb_{\Gamma_i}(\Dla(R))$. Define, for $i,j \in \cl_K$, a distribution
$\mu_i(\Psi^j) \in \Dla(\cl_K^i(p^\infty),R)$ as follows. We have a distribution $\Psi^j(\{0\}-\{\infty\})|_{(\roi_{K}\otimes_{\Z}\Zp)^\times}$
on $(\roi_K\otimes_{\Z}\Zp)^\times$. This restricts to a distribution on $(\roi_K\otimes_{\Z}\Zp)^\times/\roi_K^\times$, which gives the distribution $\mu^i(\Psi^j)$ on $\cl_K^i(p^\infty)$ under the identification above. Then define the \emph{Mellin transform} of $\Psi$ to be the ($R$-valued) locally analytic distribution on $\cl_K(p^\infty)$ given by
\[
\mel(\Psi) \defeq \textstyle\sum\limits_{i\in\cl_K} \mu_i(\Psi^i) \in \Dla(\cl_K(p^\infty),R).
\]
 A simple check identical to the arguments given in \cite[Prop. 9.7]{BW_CJM} shows that the distribution $\mel(\Psi)$ is independent of the choice of class group representatives.
\begin{definition}\label{def:non-critical Lp}
Let $\f$ be a non-critical cuspidal Bianchi eigenform of level $U_1(\n)$ with associated overconvergent class $\Psi_{\f}$. The \emph{$p$-adic $L$-function of $\f$} is the Mellin transform $L_p(\f) \defeq \mel(\Psi_{\f}) \in \Dla(\cl_K(p^\infty),L)$.
\end{definition}
Given an algebraic Hecke character $\varphi$ of $K$ of conductor $\ff = \prod_{\pri|p}\pri^{r_{\pri}} |(p^\infty)$, there is a natural associated character $\varphi_{p-\mathrm{fin}}$ of $\cl_K(p^\infty)$ (see \cite[\S7.3]{Wil17}). Let $U_{\ff} = \prod_{\pri|p}U_{\pri}^{r_{\pri}}$. The main theorem of \cite{Wil17} is the following (Thm.\ 7.4 \emph{op.\ cit}.):
\begin{theorem}\label{thm:Wil17}
For any Hecke character $\varphi$ of $K$ of conductor $\ff|(p^\infty)$ and infinity type $0 \leq (q,r) \leq (k,k)$, we have
\begin{equation}\label{eq:non-critical interpolation}
L_p(\f,\varphi_{p-\mathrm{fin}}) = \left(\textstyle\prod_{\pri|p}Z_{\pri}(\varphi)\right)A(\f,\varphi)\Lambda(\f,\varphi),
\end{equation}
for
\[ 
Z_{\pri}(\varphi) \defeq \left\{\begin{array}{cl}1-[\alpha_{\pri}\psi(\pri)]^{-1} &: \pri\nmid \ff\\
	1 &: \text{else}\end{array}\right. \hspace{3pt}\text{and}\hspace{6pt} A(\f,\varphi) \defeq  \left[\frac{\varphi(x_{\ff})d\widetilde{\tau}(\varphi^{-1})\unitsize}{(-1)^{k+q+r}2\varphi_{\ff}(x_{\ff})\alpha_{\ff}\Omega_{\f}}\right].
\]
Here $x_{\ff}$ is an explicit idele representing $\ff$, $\varphi_{\ff}$ is the restriction of $\varphi$ to $\prod_{v|\ff}K_v^\times$, $\widetilde{\tau}(\varphi^{-1})$ is the Gauss sum from \cite[\S1.2.3]{Wil17}, and $U_{\ff} \f = \alpha_{\ff}\f$. For $h_{\pri} = v_{p}(\alpha_{\pri})$, $L_p(\f)$ is $(h_{\pri})_{\pri|p}$-admissible \cite[Defs. 5.10,6.14]{Wil17}. If $\f$ has small slope, $L_p(\f)$ is unique with these interpolation and growth properties.
\end{theorem}

\begin{remark}\label{rem:twisted}
	Let $\phi$ be a finite order Hecke character of $K$. We obtain the twisted $p$-adic $L$-functions $L_p^\phi$ of the introduction by using twisted Mellin transforms. If $\phi$ has principal conductor $(c)$ prime to $p$, then for $b \in (\roi_K/c)^\times$, let $\mu_i^b(\Psi^i) \defeq \mu_i(\Psi^i|\smallmatrd{1}{b}{0}{c})$, and define $\mathrm{Mel}^\phi(\Psi) \defeq \sum_{b \in (\roi_F/c)^\times} \sum_{i\in\cl_K}\phi(b)\mu_i^b(\Psi^i)$. Then via \cite[\S3.4]{BW17}, we have $L_p^\phi(\f) = \mathrm{Mel}^\phi(\Psi_{\f})$. In general, we follow \cite[\S7.1]{Wil17}; if $\phi$ has conductor $\mathfrak{c}$, then write $\mathfrak{c} \cdot [t_i] = (c_i)\cdot [t_{j_i}]$ for $c_i \in K^\times$ and $[t_i],[t_{j_i}] \subset K$ the fractional ideals corresponding to the ideles $t_i,t_{j_i}$. Then define 
		\[
			\mathrm{Mel}^\phi(\Psi) \defeq \sum\limits_{b \in (\roi_F/\mathfrak{c})^\times} \sum\limits_{i\in\cl_K}\phi(b)\mu_i^b(\Psi^{j_i}), \hspace{12pt} \mu_i^b(\Psi^{j_i}) \defeq \mu_i\left(\Psi^{j_i}|\smallmatrd{1}{d_b}{0}{c_i}\right),
		\]
	 with $d_b \in [t_i]$ for all $i$ and $d_b \equiv b \newmod{\mathfrak{c}}$ (cf.\ \cite[\S7.1]{Wil17}). This is independent of the choice of $c_i$, and $L_p^\phi(\f) \defeq \mathrm{Mel}^\phi(\Psi_{\f}) \in \mathcal{D}(\cl_K(p^\infty),L)$.
\end{remark}

\begin{remark}\label{rem:missing cases}
	Suppose $p\roi_K = \pri\pribar$ is split. Let $\f_{\mathrm{new}}$ be base-change of weight $\lambda = (k,k)$ and  level $\mathfrak{N}$ prime to $p$. The Hecke polynomials at $\pri$ and $\pribar$ coincide; let $\alpha,\beta$ be the roots, so $\alpha+ \beta = a_{\pri}(\f_{\mathrm{new}})$ and $v_p(\alpha\beta) = k+1$. When $\alpha \neq \beta$, there are four  stabilisations $\f_{\alpha\alpha}, \f_{\alpha\beta},\f_{\beta\alpha},\f_{\beta\beta}$ to level $\mathfrak{N}p$. If $v_{p}(a_{\pri}(\f_{\mathrm{new}})) > 0$, then $0 < v_p(\alpha), v_p(\beta) <k+1$, so each stabilisation is small slope, giving four $p$-adic $L$-functions attached to $\f$. If $k = 0$, these are the $p$-adic $L$-functions of \cite[\S5]{Loe14}. If $v_{p}(a_{\pri}(\f)) = 0$, then take $v_p(\alpha) = 0, v_p(\beta) = k+1$; then only $f_{\alpha\alpha}$ has small slope, and \cite{Wil17} does not give $p$-adic $L$-functions for $\f_{\alpha\beta},\f_{\beta\alpha}$ or $\f_{\beta\beta}$ .
	
	For $p$ inert, the Hecke roots $\alpha,\beta$ satisfy $v_{p}(\alpha\beta) = 2(k+1)$; so at least one of $v_{p}(\alpha),v_{p}(\beta)$ is $\geq k+1$, and there is \emph{always} a missing $p$-adic $L$-function.
\end{remark}

%
%
\section{The Bianchi eigenvariety}
We summarise results on the Bianchi eigenvariety, following \cite{Han17}. Hansen's results are stated for singular cohomology, but in \cite[\S3.3]{Han17} he gives tools to produce (identical) proofs for cohomology with compact support.

\subsection{Distributions over the weight space}
\begin{definition}\label{def:weight space}
The \emph{Bianchi weight space of level $\n$} is the rigid analytic space $\W_{K,\n}$ whose $L$-points, for $L\subset \Cp$ any sufficiently large extension of $\Qp$, are 
\[\W_{K,\n}(L) = \mathrm{Hom}_{\mathrm{cts}}((\roi_K\otimes_{\Z}\Zp)^\times/E(\n),L^\times),\]
where $E(\n) \defeq\{\epsilon \in \roi_K^\times: \epsilon \equiv 1 \newmod{\n}\}.$ This can be identified with $Z(\GL_2(K)) \cap U_1(\n)$, hence this invariance ensures the existence of non-trivial `weight $\lambda$' local systems on $Y_1(\n)$. Since the level will typically be clear from context, we will usually drop the subscript $\n$ from the notation. 
\end{definition}

A weight $\lambda \in \W_K(L)$ is \emph{classical} if it can be written in the form $\epsilon \lambda^{\mathrm{alg}}$, where $\epsilon$ is a finite order character and $\lambda^{\mathrm{alg}}(z) = z^{\mathbf{k}} = z^{k_1}\overline{z}^{k_2}$, where $\mathbf{k} = (k_1,k_2) \in \Z_{\geq 0}^2$.
\begin{remark}
	This is a slightly smaller space of `null weights' than that considered in Hansen, who uses the characters on the torus $T(\Zp)$ of diagonal matrices in $\GL_2(\roi_K\otimes_{\Zp}\Zp)$. The two spaces are essentially the same after twisting by a power of the norm, and the smaller space allows clearer comparison with the Coleman--Mazur eigencurve. 
\end{remark}

For each $\lambda \in \W_K(L)$, as before one can define a weight $\lambda$ action of $\Sigma_0(p)$ on $\mathcal{A}(L)$ by $\gamma\cdot_{\lambda}f(z) = \lambda(a+cz)f\left(\frac{b+dz}{a+cz}\right),$ and hence a dual action on $\Dla(L)$. We can vary these action in families over $\W_K$. Let $\Omega \subset \W_K$ be an affinoid, equipped with a tautological character
\[\chi_\Omega : (\roi_K\otimes_{\Z}\Zp)^\times \longrightarrow \roi(\Omega)^\times,\]
such that for any $\lambda\in \Omega(L),$ the homomorphism $\lambda : (\roi_K\otimes_{\Z}\Zp)^\times \rightarrow L^\times$ factors as 
$(\roi_K\otimes_{\Z}\Zp)^\times \rightarrow \roi(\Omega)^\times \rightarrow L^\times,$
where the second map is evaluation at $\lambda$. We can thus equip $\laA_\Omega \defeq \laA(\roi(\Omega))$ with a `weight $\Omega$' action of $\Sigma_0(p)$ given by
\begin{equation}\label{def:A family}
\gamma\cdot_{\Omega} f(z) = \chi_\Omega(a+cz)f\left(\textstyle\frac{b+dz}{a+cz}\right).
\end{equation}
Dually we get an action on $\Dla_\Omega \defeq \Dla(\roi(\Omega))$, giving a local system $\DD_\Omega$ on $\Y$.

If $\Sigma \subset \Omega$ is a closed subset, then $\Dla_\Omega \otimes_{\roi(\Omega)}\roi(\Sigma) \cong \Dla_\Sigma$ (see \cite[\S2.2]{Han17}). In particular, if $\lambda \in \Omega(L)$ corresponds to a maximal ideal $\m_\lambda \subset \roi(\Omega)$, then $\Dla_\Omega\otimes_{\roi(\Omega)}\roi(\Omega)/\m_\lambda \cong \Dla_\lambda(L)$.

\subsection{The eigenvariety and base-change functoriality}
One of the main results of \cite{Han17} specialises, in our setting, to the following. Recall from above that his results apply also to compactly supported cohomology, and recall $\HH_{\n,p}$ from Def.\ \ref{def:hecke alg}. We write $\h_{\mathrm{c}}^{*}$ for total cohomology.

\begin{theorem}[Hansen]\label{t:hansen's eigencurve}
There exists a separated rigid analytic space $\E_{\n}$, and a morphism $w: \E_{\n} \rightarrow \W_K$, such that for each finite extension $L$ of $\Q_p$, the $L$-points $y$ of $\E_{\n}$ with $w(y) = \lambda \in \W_K(L)$ are in bijection with systems $\psi_y : \HH_{\n,p} \to L$ of Hecke eigenvalues occurring in $\h_{\mathrm{c}}^*(\Y, \DD_{\lambda}(L))$. 
\end{theorem}

The level $\n$ will often be clear from context, so we usually drop the subscript.

Let $\W_{\Q}$ denote the (null) weight space for $\GLt/\Q$, that is, the rigid analytic space whose $L$-points are $\W_{\Q}(L) = \mathrm{Hom}_{\mathrm{cts}}(\Zp^\times,L^\times)$ for $L \subset \Cp$. For the rest of the paper, let $\CC$ be the Coleman--Mazur eigencurve of tame level $\Gamma_1(N')$, for $N'$ the prime-to-$p$ part of $N$. Abusing notation, we call the weight map $w : \CC \to \W_{\Q}$. 
	
	A point $y \in \E_{\n}$ (resp. $\CC$) is \emph{classical} if there is a Bianchi (resp.\ classical) eigenform $\f_y$ (resp.\ $f_y$) of weight $w(y)$ such that $t\f_y = \psi_y(t)\f_y$ (resp.\ $tf_y = \psi_y(t)f_y$) for all $t \in \HH_{\n,p}$ (resp.\ for all classical Hecke operators $t$). A classical point $y$ satisfies Conditions \ref{running assumptions} if $\f_y$ does (resp.\ Conditions \ref{running assumptions Q} if $f_y$ does).

 There is (for any level $\n$) a closed immersion $\W_\Q \hookrightarrow \W_{K,\n}$ induced by the norm map $(\roi_K\otimes_{\Z}\Zp)^\times \rightarrow \Zp^\times$. Then \cite[Thm.\ 5.1.6]{Han17} or \cite[Thm.\ 3.2.1]{JoNew} combined with \cite[\S4.3]{JoNew} gives:

\begin{theorem} \label{thm:base-change}
There is a finite morphism $\mathrm{BC}_N: \CC \longrightarrow \E_{N\roi_K}$ of rigid spaces interpolating base-change functoriality on classical points. More precisely, if $y \in \CC(L)$ corresponds to a classical modular form $f$, then $\mathrm{BC}_N(y) \in \E(L)$ corresponds to the (stabilisation to level $N\roi_K$ of the) system of eigenvalues attached to the base-change of $f$ to $\GLt/K$.
\end{theorem}

This is based on an idea of Chenevier \cite{Che05}. We actually require a refined version of this result, defined locally, giving more precise control over the level. Let $x$ be a classical point in $\CC$ satisfying Conditions \ref{running assumptions Q}, corresponding to a level $N$ eigenform $f_x$. If the level of its base-change $\f_x$ is not $N\roi_K$, i.e.\ if the level drops under base-change, then $\mathrm{BC}_N(x) \in \E_{N\roi_K}$ corresponds to a Bianchi form that does \emph{not} satisfy (C2) of Conditions \ref{running assumptions}: it has also been stabilised at some other $\mathfrak{q}\nmid p$ dividing $N\roi_K$. Locally at such $x \in \CC$, however, it is possible to define a $p$-adic base-change map $\mathrm{BC}_{\n}$ that does send points satisfying Conditions \ref{running assumptions Q} to points satisfying Conditions \ref{running assumptions}.

\begin{proposition}\label{prop:zd pstab}
	Let $x \in \CC$ be a classical point satisfying Conditions \ref{running assumptions Q}.  Then there exists a neighbourhood $V_{\Q} \subset \CC$ of $x$, an ideal $\n \subset \roi_K$ divisible by each $\pri|p$ and with $\tfrac{N}{(N,d)}\roi_K | \n | N\roi_K$, and a finite morphism $\mathrm{BC}_{\n} : V_{\Q} \to \E_{\n}$ that interpolates base-change functoriality on classical points and such that every classical point in $\mathrm{BC}_{\n}(V_{\Q})$ satisfies Conditions \ref{running assumptions}.
\end{proposition}
\begin{proof}	
	We closely follow the strategy of \cite[Lem.\ 2.7]{Bel12}. Let $V_{\Q}$ be a neighbourhood of $x$ in $\CC$.  Let $\rho_{V_{\Q}} : G_{\Q} \to \GL_2(\roi(V_{\Q}))$ be the attached Galois representation over $V_{\Q}$ (from e.g.\ \cite[\S4]{BC09}) where $G_{\Q} = \mathrm{Gal}(\overline{\Q}/\Q)$. For a classical point $y \in V_{\Q}$, the tame level of $f_y$ depends only on  $[\rho_{V_\Q}|_{I_\ell}](y)$ for $\ell \neq p$, where $I_\ell$ is the inertia subgroup at $\ell$ (see \cite{SerreConductors}) defined using the embedding $\overline{\Q} \hookrightarrow \overline{\Q}_\ell$. Note $\rho_{V_{\Q}}|_{I_\ell}$ is trivial unless $\ell|N$. By considering the attached family of Weil--Deligne representations, one sees that the conductor of $\rho_{V_{\Q}}|_{I_{\ell}}$ is locally constant at $x$ (see \cite[Thm.\ 3.1]{SahaConductors}, noting that the Weil--Deligne representation attached to the classical cuspidal point $x$ is pure). Hence we may shrink so that the conductor at each $\ell|N$ -- and hence the tame level -- is constant over $V_{\Q}$. Then every classical point in $V_{\Q}$ satisfies Conditions \ref{running assumptions Q}.
	
	The Galois representation attached to $\mathrm{BC}_N(V_{\Q}) \subset \cE_{N\roi_K}$ is $\rho_{V_{\Q}}|_{G_K}$, where $G_K = \mathrm{Gal}(\overline{K}/K)$. The tame level of $z \in \mathrm{BC}_N(V_{\Q})$ similarly depends only on $[\rho_{V_{\Q}}|_{I_{K,\ell}}](z)$, where $I_{K,\ell} = I_\ell \cap G_K$. Thus the tame level is constant over $\mathrm{BC}_N(V_{\Q}) \subset \cE_{N\roi_K}$. Thus if $f_x$ base-changes to level $\n$, then so does $f_y$ for every nearby classical $y \in \CC$. Applying \cite[Thm.\ 3.2.1]{JoNew} we get the required map $V_{\Q} \to \E_{\n}$.
\end{proof}

\begin{remark}\label{rem:zd pstab}
	For clarity of argument, in the remainder of the paper, we will assume that if $x \in \CC$ satisfies Conditions \ref{running assumptions Q}, then there is a neighbourhood $V_{\Q}$ of $x$ in $\CC$ such that every classical point of $\mathrm{BC}_N(V_{\Q})\subset \E_{N\roi_K}$ satisfies Conditions \ref{running assumptions}. This is always the case, for example, if the tame level of $x$ is coprime to $d$. Since the proofs in the sequel are all local in nature, all of the results can be proved without this assumption by working in $\E_{\n}$ for some $\n|N\roi_K$ and using Prop.\ \ref{prop:zd pstab}.  We shall henceforth always drop the $N$ and $\n$ and just write $\mathrm{BC}$ and $\E$.
\end{remark}

\subsection{The dimensions of irreducible components} 

\begin{proposition}\label{prop:point}
Let $\f \in S_\lambda(U_1(\n))$ be a finite slope cuspidal Bianchi eigenform. There is a point $x_{\f} \in \E(L)$ corresponding to $\f$.
\end{proposition}
\begin{proof}
If $\f$ is non-critical, then there are eigenclasses $\Psi^i \in \h_{\mathrm{c}}^i(\Y, \DD_{\lambda}(L))$  for $i = 1,2$ corresponding to $\f$, and hence a point $x_{\f} \in \E(L)$. If $\f$ is critical, then consider the long exact sequence of cohomology attached to $\Dla_\lambda \rightarrow V_\lambda^*$. The cokernel of the map $\rho_2: \h^2_{\mathrm{c}}(\Y,\DD_\lambda(L)) \rightarrow \h^2_{\mathrm{c}}(\Y,\VV_\lambda(L)^*)$ can be identified as a subspace of a degree 3 overconvergent cohomology group (see \cite[\S9.3]{BW_CJM}); but an analysis as in \cite[Lem.\ 3.9]{Bel12} shows that cuspidal eigensystems do not appear in such spaces. In particular, after restricting to the generalised eigenspace at $\f$, the map $\rho_2$ is surjective. Thus $\h_{\mathrm{c}}^{*}(\Y,\DD_\lambda(L))_{(\f)} \neq 0$, and there exists a Hecke eigenclass with the same eigenvalues as $\f$ in this space, as required.
\end{proof}
For our purposes, if $\f$ is critical it suffices to assume $\f$ is base-change, whence such a point $x_{\f}$ arises in the image of $\mathrm{BC}$.
\begin{theorem}[Hida, Hansen--Newton]\label{thm:dimension 1}
Suppose $\f$ is non-critical. Any irreducible component $\cI$ of $\E$ passing through $x_{\f}$ has dimension 1.
\end{theorem}
\begin{proof}
In the ordinary case, this is due to Hida \cite{Hid95}. In general, \cite[Prop.\ B.1]{Han17} shows that $\cI$ has dimension at least 1. The following was pointed out to us by Hansen. Suppose $\cI$ is a 2-dimensional irreducible component passing through $x_{\f}$, and let $\Omega = w(\cI)$. Let $\rho_{\cI}$ be the two-dimensional Galois pseudocharacter $\rho_{\cI}$ over $\cI$ of \cite{JoNewExt}. Let $\cI_{\mathrm{ss}}$ be the set of points $y$ of $\cI$ such that $w(y)$ is classical non-parallel in $\Omega$ and $y$ has small slope.
This set is Zariski-dense in $\cI$. Each $y \in \cI_{\mathrm{ss}}$ necessarily corresponds to a classical form by \cite[Thm.\ 8.7]{BW_CJM}, and this classical form must be Eisenstein, as classical cuspidal forms exist only at parallel weights. Hence the specialisation of $\rho_{\cI}$ at $y$ is reducible. Reducibility is a Zariski-closed condition, so $\rho_{\cI}$ and its specialisation at $x_{\f}$ are reducible. As $\f$ is cuspidal, its attached Galois representation is irreducible, so we get a contradiction.
\end{proof}

%
%

\section{Families of modular symbols}\label{sec:families of ms}
In \S3, we gave results about $p$-adic families using the total cohomology. The $p$-adic $L$-functions of  \cite{Wil17} arise \emph{only} from $\h_{\mathrm{c}}^1$, however, so in this section we refine the above to show that $p$-adic families can be realised in $\hc$ through modular symbols. Unlike in classical settings, this is obstructed by the contribution of a cuspidal Bianchi eigenform $\f$ to classical cohomology in degrees 1 \emph{and} 2. Counter-intuitively, a key step in overcoming this is the `purity' result of Lem.\ \ref{lem:min degree}, which implies that $\f$ appears in $\h_{\mathrm{c}}^i(Y_1(\n),\mathscr{D}_\Omega)$ -- for two-dimensional affinoids $\Omega$ -- only for $i=2$. This control allows us to isolate families of modular symbols (in $\hc$) over certain curves in $\Omega$ and prove that they are free of rank one over a Hecke algebra.

\subsection{Families in $\hc(\Y,\DD_\Sigma)$}
A \emph{slope-adapted affinoid} is a pair $(\Omega,h)$, where $\Omega = \mathrm{Sp}(\cO(\Omega))\subset \W_K$ is a two-dimensional affinoid in weight space and $h \in \Q_{\geq 0}$ such that there exists a $U_p$-slope decomposition (as in \cite[\S2.3.1]{Urb11}), stable under the action of $\HH_{\n,p}$ from Def.\ \ref{def:hecke alg}:
\[
\h_{\mathrm{c}}^*(\Y,\DD_{\Omega}) \cong \h_{\mathrm{c}}^*(\Y,\DD_{\Omega})\ssh \oplus \h_{\mathrm{c}}^*(\Y,\DD_{\Omega})^{>h}.
\]
\begin{definition}
For a slope-adapted affinoid $(\Omega,h)$, let
\[\TT_{\Omega,h} \defeq \text{Image of }\HH_{\n,p}  \otimes \cO(\Omega) \text{ in }\mathrm{End}_{\roi(\Omega)}(\h_{\mathrm{c}}^*(\Y,\DD_{\Omega})\ssh)\]
(using \emph{total} cohomology).  We define a local piece of the eigenvariety by
\[\E_{\Omega,h} \defeq \mathrm{Sp}(\TT_{\Omega,h}).\]
\end{definition}

 The affinoids $\E_{\Omega,h}$ glue to give $\E$. We have a bijection between eigensystems $\psi_x : \HH_{\n,p} \to L$ arising in $\h_{\mathrm{c}}^*(\Y,\DD_{\lambda}(L))\ssh$ and $x \in \E_{\Omega,h}(L)$ with $w(x) = \lambda \in \Omega(L)$. Attached to such $x$ is a maximal ideal $\m_x \subset \TT_{\Omega,h}$; pulling back gives a maximal ideal in $\HH_{\n,p}\otimes \cO(\Omega)$. This maximal ideal is generated by $\m_\lambda$ and $t - \psi_x(t)$ for all $t \in \HH_{\n,p}$, where $\m_\lambda$ is the maximal ideal of $\cO(\Omega)$ attached to $\lambda$. 
 	
 	The following is an unpublished result of David Hansen.

\begin{lemma}\label{lem:min degree}
	\begin{itemize}
		\item[(i)]
		The spaces $\h_{\mathrm{c}}^0(\Y,\DD_\lambda)\ssh$ and $\h_{\mathrm{c}}^0(\Y,\DD_\Omega)\ssh$ are both 0.
		\item[(ii)]
		Let $x$ be a cuspidal classical point of $\E_{\Omega,h}$. The system of eigenvalues for $x$ occurs in $\h^i_{\mathrm{c}}(\Y,\DD_{\Omega})\ssh$ if and only if $i=2$.
	\end{itemize}
\end{lemma}
\begin{proof}
	(i) Recall $\cA(R)$ from Def.\ \ref{def:A}. Let $\cA^0(R) \subset \cA(R)$ be the subspace of rigid analytic functions, and let $\Dz = \Dz(R) = \mathrm{Hom}_{\mathrm{cts}}(\rigidA(R),R)$. For $R = L$ or $\roi(\Omega)$, these spaces carry actions of $\Sigma_0(p)$ exactly as in Def.\ \ref{def:A} and \eqref{def:A family} respectively, and we thus get attached local systems $\DD^0$ on $Y_1(\n)$.
	
	 We first prove that $\h_{\mathrm{c}}^0(\Y,\DD^0) = 0$. For singular cohomology, we have $\h^0(\Y,\Dl) =  \oplus_{i \in \cl_K} \h^0(\Gamma_i\backslash \uhs, \Dl) = \oplus_{i \in \cl_K}\h^0(\Gamma_i,\Dl) = \oplus_{i \in \cl_K}(\Dla^0)^{\Gamma_i},$ where $\Gamma_i$ are as defined in \cite[Def.\ 3.2]{Wil17}. For $b \in \roi_K$, let $\gamma_b \defeq \smallmatrd{1}{b}{0}{1}$, which acts on $\rigidA$ by sending $f(z)$ to $f(z+b)$. For each $i$, there is an ideal $I_i \subset \roi_K$ such that $\{\gamma_b : b \in I_i\} \subset \Gamma_i$ (see \cite[Def.\ 3.2]{Wil17}). Fix $i \in \cl_K$ and $b \in I_i$, and let $\mu \in (\Dz)^{\Gamma_i}$; then $\mu(z\mapsto z) = \mu|\gamma_b(z \mapsto z) = \mu(z\mapsto z) + \mu(z \mapsto b),$ so that $\mu(z\to b) = 0$ and hence $\mu$ is zero on the constant functions. Suppose $\mu$ is zero on functions that are polynomial of degree less than $r-1$. Then consider any monomial $z \mapsto \theta_{j,r+1}(z) \defeq z^j\overline{z}^{r+1-j}$ of degree $r+1$. We have
	\begin{align*}
		\mu(\theta_{j,r+1}) = \mu|\gamma_b(\theta_{j,r+1}) &= \mu(z\mapsto (z+b)^{j}(\overline{z}+\overline{b})^{r+1-j})\\
		&= \mu(\theta_{j,r+1}) + bj\mu(\theta_{j-1,r}) + \overline{b}(r+1-j)\mu(\theta_{j,r})
	\end{align*}
	for all $b \in I_i$, where the lower terms vanish by assumption. Taking $0 \neq b \in \Z$ gives $j\mu(\theta_{j-1,r}) + (r+1-j)\mu(\theta_{j,r}) = 0$; and taking $0 \neq b \in \sqrt{-d}\Z$ gives $j\mu(\theta_{j-1,r}) - (r+1-j)\mu(\theta_{j,r}) = 0$. Solving, we conclude that $\mu(\theta_{j-1,r}) = \mu(\theta_{j,r}) = 0$, and since we can work with arbitrary $0 \leq j \leq r$, we conclude that $\mu$ vanishes on all monomials of degree $r$. Thus $\mu = 0$ by induction, so $(\Dla^0)^{\Gamma_i} = 0$, and hence $\h^0(\Y,\DD^0) = 0$. Then $\h^0_{\mathrm{c}}(\Y,\DD^0) = 0$ since the excision exact sequence for the Borel--Serre compactification of $\Y$ starts $0 \rightarrow \h_{\mathrm{c}}^0 \rightarrow \h^0$. After passing to the small slope parts, overconvergent cohomology with coefficients in $\Dla^0$ and $\Dla$ agree \cite[Lem.\ 2.3.13]{Urb11}, proving (i).
	
	(ii) We first claim that $x$ does not appear as an eigensystem in $\h_{\mathrm{c}}^3$, for which we follow \cite[Lem.\ 5.2]{PS12}. We identify $\h_{\mathrm{c}}^3(\Y,\DD_\Omega) \cong \h_0(\Y,\DD_\Omega)$ using Poincar\'e duality. This decomposes into a direct sum $\oplus_{i \in \cl_K} \Dla_\Omega/\Gamma_i\Dla_\Omega$, and an analysis as \emph{op.\ cit}.\ shows the only eigensystems supported on this module are attached to overconvergent weight $(0,0)$ Eisenstein series. They are thus not cuspidal, and $x$ does not appear in $\h_{\mathrm{c}}^3$. 
	
	It remains to show $x$ does not appear in $\hc(\Y,\DD_\Omega)$. We exploit Hansen's Tor spectral sequence
	\[
	E_2^{i,j} = \Tor_{-i}^{\roi(\Omega)}(\h^j_{\mathrm{c}}(\Y,\DD_\Omega)\ssh,L) \implies \h^{i+j}_{\mathrm{c}}(\Y,\DD_\lambda)\ssh,
	\]
	where $\m_\lambda$ is any maximal ideal of $\roi(\Omega)$ and $L = \roi(\Omega)/\m_\lambda$. Since $\roi(\Omega)$ is regular of dimension 2, the $\Tor_i^{\roi(\Omega)}$ groups vanish for $i \geq 3$, so that $E_2^{-3,2} = 0$. As $E_2^{1,0} = 0$ as well, we see that 
	\[
	E_3^{-1,1} = \ker(E_2^{-1,1} \rightarrow 0)/\mathrm{Image}(0 \rightarrow E_2^{-1,1}) = E_2^{-1,1},
	\]
	and continuing, that $E_\infty^{-1,1} = \Tor_1^{\roi(\Omega)}(\hc(\Y,\DD_\Omega)\ssh,L)$. This contributes to the grading on $\h_{\mathrm{c}}^0(\Y,\DD_\lambda)\ssh$, which is zero by the above; hence this Tor term vanishes. A similar analysis, using that $E_2^{0,0} = E_2^{-4,2} = 0$, shows that 
	\[
	E_\infty^{-2,1} = E_2^{-2,1} = \Tor_2^{\roi(\Omega)}(\hc(\Y,\DD_\Omega)\ssh,L) = 0
	\]
	 as well. Then $\Tor_i^{\roi(\Omega)}(\hc(\Y,\DD_\Omega)\ssh,\roi(\Omega)/\m_\lambda) = 0$ for all $i > 0$, and for any maximal ideal $\m_\lambda$, so by \cite[Prop.\ A.3]{Han17}, the $\roi(\Omega)$-module $\hc(\Y,\DD_\Omega)\ssh$ is either zero or projective. As it is torsion by \cite[Thm.\ 4.4.1]{Han17}, it cannot be projective, so it vanishes, as required.
\end{proof}

Despite this purity in degree 2, we now show how to exhibit families of Bianchi modular forms in $\hc$, where our constructions of $p$-adic $L$-functions take place.

\begin{definition}
\label{def:family notation}
Let $x \in \E_{\Omega,h}(L)$ be any point corresponding to a maximal ideal $\m_x$ in $\TT_{\Omega,h}$. Let $\mathscr{P}_x$ be a minimal prime of $\TT_{\Omega,h}$ contained in $\m_x$, and write $\mathscr{P}_\lambda$ for the contraction of $\mathscr{P}_x$ to $\roi(\Omega)$. Define $\Lambda = \roi(\Omega)/\mathscr{P}_\lambda$ and let $\Sigma = \mathrm{Sp}(\Lambda)$ be the corresponding closed subset inside $\Omega$, which is a rigid curve by Thm.\ \ref{thm:dimension 1}. If such a curve $\Sigma \subset \Omega$ arises in this way, we say that \emph{$x$ varies in a family over $\Sigma$}.
\end{definition}

\begin{proposition}\label{prop:appear in H1}
Let $x$ be a cuspidal classical point of $\E_{\Omega,h}(L)$ that varies in a family over $\Sigma$. Then, after possibly shrinking $\Sigma$,
\[\hc(\Y,\DD_{\Sigma})\ssh\locx \neq 0.\]
\end{proposition}

\begin{proof}
Given Lem.\ \ref{lem:min degree}, the result now follows from \cite[Lem.\ B.3]{Han17} combined with the final remark of Appendix B \emph{op.\ cit}. We sketch a direct proof. The ideal $\mathscr{P}_\lambda$ has height one, and still has height one in the localisation $\roi(\Omega)\locl$. This localisation is a regular local ring, and hence a unique factorisation domain, so all height one primes are principal, and we can take some generator $r$ of $\mathscr{P}_\lambda\roi(\Omega)\locl$. After possibly shrinking $\Omega$, and scaling by a unit in $\roi(\Omega)\locl$, we may assume $r \in \roi(\Omega)$. We obtain a short exact sequence $0 \rightarrow \Dla_\Omega \rightarrow \Dla_\Omega \rightarrow \Dla_{\Sigma} \rightarrow 0$, where the first map is multiplication by $r$. By truncating the associated long exact sequence at the first degree 2 term, and localising at $x$, we obtain a short exact sequence
\[
\hc(\Y,\DD_\Omega)\ssh\locx \rightarrow \hc(\Y,\DD_{\Sigma})\ssh\locx \rightarrow \h_{\mathrm{c}}^2(\Y,\DD_\Omega)\ssh\locx[r] \rightarrow 0.
\]
By Lem.\ \ref{lem:min degree}, the first term vanishes, so the second map is an isomorphism. Minimal primes of $\TT_{\Omega,h}$ are in bijection with associated primes in $\h_{\mathrm{c}}^2(Y_1(\n), \mathscr{D}_\Omega)^{\leq h}$ by \cite[Thm.\ 6.5]{Mat89} and Lem.\ \ref{lem:min degree}, so the system of eigenvalues corresponding to $x$ is $\mathscr{P}_x$-torsion. Thus $\mathscr{P}_\lambda$ is an associated prime of $\h_{\mathrm{c}}^2(Y_1(\n),\mathscr{D}_\Omega)\ssh_{\m_x}$, i.e.\ it annihilates a non-zero element. Thus $\h_{\mathrm{c}}^2(Y_1(\n),\mathscr{D}_\Omega)^{\leq h}_{\m_x}[r]$ is non-zero, from which we conclude.
\end{proof}

We were unable to find a proof of this proposition that used only the short exact sequence $0 \to \Dla_\Sigma \to \Dla_\Sigma \to \Dla_\lambda \to 0$, due to the presence of classes in both degree 1 and 2; the additional input from $\Dla_\Omega$, via the `purity' of Lem.\ \ref{lem:min degree}, appears to be necessary to obtain sufficient control.

Recall $\m_x \subset \HH_{\n,p}\otimes \cO(\Omega)$; we have a maximal ideal $\m_x \otimes_{\cO(\Omega)} \Lambda \subset \HH_{\n,p} \otimes \Lambda$. Let
\[\TT_{\Sigma,h} \defeq \text{image of }\HH_{\n,p}\otimes \Lambda \text{ in End}_\Lambda(\hc(\Y,\DD_\Sigma)\ssh).\]
By Prop.\ \ref{prop:appear in H1}, the image of $\m_x \otimes \Lambda \subset \HH_{\n,p}\otimes \Lambda$ in $\TT_{\Sigma,h}$ is maximal; abusing notation, we also denote this by $\m_x$. Thus $\m_x$ corresponds to a point $x \in \Sp(\TT_{\Sigma,h})$.

\subsection{A structure theorem}
Let $x \in \E_{\Omega,h}(L)$ correspond to a cuspidal non-critical classical Bianchi eigenform $\f$, varying in a family over a curve $\Sigma \subset \Omega$. Let $\lambda = w(x)$. By Thm.\ \ref{thm:mult one}, the following holds if $\f$  satisfies Conditions \ref{running assumptions}:
\begin{equation}\label{eqn:mult one}\text{ The $\HH_{\n,p}$-generalised eigenspace $\hc(\Y,\VV_\lambda(L)^*)_{(\f)}$ is one-dimensional.}
\end{equation}

\begin{theorem}\label{thm:free rank 1}
Let $\f$ be non-critical satisfying (\ref{eqn:mult one}). Suppose $\Sigma$ is smooth at $\lambda$. Then the space $\hc(\Y,\DD_\Sigma)\ssh\locx$ is free of rank 1 over $(\TT_{\Sigma,h})\locx$, which (after replacing $\Lambda$ with $\Lambda\otimes_{\Qp}L$) is itself free of rank 1 over $\Lambda_{\m_{\lambda}}$.
\end{theorem}
Note that smoothness of $\Sigma$ is satisfied in base-change components. We give the proof after two lemmas, and thank Adel Betina, who contributed to Lem.\ \ref{prop:gen by 1}.
\begin{lemma}\label{prop:gen by 1}
The $\Lambda_{\m_{\lambda}}$-module $\hc(\Y,\DD_\Sigma)\ssh\locx$ is generated by one element.
\end{lemma}

\begin{proof}
The localisation $\hc(\Y,\DD_\Sigma)\ssh\locx$ is a finite $\Lambda\locl$-module by general facts on slope decompositions.  Thus we may freely use Nakayama's lemma.

From the short exact sequence of distribution spaces given by the natural surjection $\mathrm{sp}_{\lambda}: \Dla_\Sigma \rightarrow \Dla_\lambda$, we obtain a long exact sequence of cohomology, which (since $\m_\lambda$ is principal by smoothness) we truncate to an exact sequence
\[
0 \rightarrow \hc(\Y,\DD_\Sigma)\ssh \otimes_\Lambda \Lambda/\m_\lambda \rightarrow \hc(\Y,\DD_\lambda)\ssh.
\]
Since localising is exact, we deduce the existence of an exact sequence
\begin{equation}\label{eq:injection}
0 \rightarrow \hc(\Y,\DD_\Sigma)\locx\ssh \otimes_{\Lambda\locl} \Lambda\locl/\m_\lambda \rightarrow \hc(\Y,\DD_\lambda)\locx\ssh.
\end{equation}
The last term is the generalised eigenspace corresponding to the system of eigenvalues attached to $x$. (At this point, we are assuming that we have extended the base field of $\Lambda$ so that $x$ is defined over $\Lambda/\m_\lambda$). As $x$ is non-critical, this is isomorphic to the generalised eigenspace of $\f$ for $\bH_{\n,p}$ in the classical cohomology, and 
by assumption \eqref{eqn:mult one}, this is one-dimensional. Suppose the first term is 0; then by Nakayama's lemma, we must have $\hc(\Y,\DD_\Sigma)\ssh\locx = 0,$ which contradicts Prop.\ \ref{prop:appear in H1}. Hence the first term is one-dimensional and there is an isomorphism
 \[
  \hc(\Y,\DD_\Sigma)\ssh\locx\otimes_{\Lambda\locl}\Lambda\locl/\m_\lambda \cong \hc(\Y,\DD_\lambda)\ssh\locx.
 \]
Now we use Nakayama again. A generator of $\hc(\Y,\DD_\Sigma)\ssh\locx \otimes_{\Lambda\locl}\Lambda\locl/\m_\lambda$ lifts to a generator of $\hc(\Y,\DD_\Sigma)\ssh\locx$ over $\Lambda\locl$, which completes the proof.
\end{proof}

\begin{lemma}\label{lem:torsion free}
The $\Lambda$-module $\hc(\Y,\DD_\Sigma)\ssh$ is projective.
\end{lemma}
\begin{proof}
We use the identification with modular symbols. For fixed $i$, let $\{\delta_j : j\in J\}$ be a finite set of generators for $\Delta_0$ as a $\Z[\Gamma_i]$-module \cite[Lem.\ 3.8]{Wil17}. For any $R$, the map $\symb_{\Gamma_i}(\Dla(R)) \hookrightarrow \Dla(R)^J, \Psi \mapsto (\Psi(\delta_j))_{j\in J}$ is an injective $R$-module map. By passing to the direct sum over all $i \in \cl_K$, we obtain a $\Lambda$-module embedding of $\hc(\Y,\DD_\Sigma)$ into a finite direct sum of copies of $\Dla_\Sigma$. But $\Dla_\Sigma$ is a torsion-free $\Lambda$-module since $\Lambda$ is a domain. Hence $\hc(\Y,\DD_\Sigma)\ssh$ is finite torsion-free over $\Lambda$. Thus $\mathrm{Tor}_i^\Lambda(\hc(\Y,\DD_\Sigma)\ssh,\Lambda/\m_\lambda) = 0$ for all $i > 0$ and $\lambda \in \Sigma$; for $i = 1$, this is by torsion-freeness, and for $i \geq 2$, this follows by smoothness of $\Sigma$ (so $\Lambda$ is regular of dimension 1) and \cite[Thm.\ 19.2]{Mat89}. We conclude by \cite[Prop.\ A.3]{Han17}.
\end{proof}

\begin{proof} \textbf{(Thm.\ \ref{thm:free rank 1})}. By Lems.\ \ref{prop:gen by 1} and \ref{lem:torsion free}, $\hc(\Y,\DD_\Sigma)\ssh\locx$ is free of rank 1 over $\Lambda\locl$. As
 \[
   (\TT_{\Sigma,h})\locx \subset \mathrm{End}_{\Lambda\locl}\left(\hc(\Y,\DD_\Sigma)\ssh\locx\right) \cong \Lambda\locl
 \]
is non-zero by our assumption on $\Sigma$, we must have $(\TT_{\Sigma,h})\locx \cong \Lambda\locl$. As the actions of $\TT_{\Sigma,h}$ and $\Lambda$ on $\hc(\Y,\DD_\Sigma)\ssh$ are compatible, we see $\hc(\Y,\DD_\Sigma)\ssh\locx$ is free of rank 1 over $(\TT_{\Sigma,h})\locx$, completing the proof of Thm.\ \ref{thm:free rank 1}.
\end{proof}


\begin{corollary}\label{prop:free neighbourhood}
Possibly shrinking $\Sigma$, there exists a connected component $V = \Sp T\subset \Sp(\TT_{\Sigma,h})$ containing $x$ such that $\hc(\Y,\DD_\Sigma)\ssh\otimes_{\TT_{\Sigma,h}}T$ is free of rank one over $T$, which is free of rank one over $\Lambda$. Thus the weight map $V \to \Sigma$ is \'etale.
\end{corollary}
\begin{proof}
(Compare \cite[Lem.\ 2.10]{BDJ17}). Define \emph{rigid analytic localisations} by
\[
\Lambda_\lambda = \varinjlim_{\lambda \in U\subset\Sigma}\roi(U),\hspace{20pt}
(\TT_{\Sigma,h})_x = \varinjlim_{x \in V \subset \Sp(\TT_{\Sigma,h})}\roi(V),
\]
\[
\hc(\Y,\DD_\Sigma)\ssh_x = \textstyle\varinjlim_{x \in V \subset \Sp(\TT_{\Sigma,h})} \hc(\Y,\DD_\Sigma)\ssh\otimes_{\TT_{\Sigma,h}}\roi(V).
\]
By \cite[\S7.3.2,\S7.3.3]{BGR}, $\Lambda_\lambda$ and $(\TT_{\Sigma,h})_x$ are faithfully flat extensions of $\Lambda_{\m_\lambda}$ and $(\TT_{\Sigma,h})_{\m_x}$ respectively, with isomorphic completions; and combining \cite[\S7.3.3, Prop.\ 4]{BGR} with Thm.\ \ref{thm:free rank 1}, we see $\hc(\Y,\DD_\Sigma)\ssh_x$ is free of rank one over $(\TT_{\Sigma,h})_x$, which is free of rank one over $\Lambda_\lambda$. Thus we are in the situation of \cite[Lem.\ 2.10]{BDJ17} (over the rigid space $\Sigma$), so that -- possibly shrinking $\Sigma$ -- we may choose $V \subset \Sp(\TT_{\Sigma,h})$ such that $T = \roi(V)$ is free of rank one over $\Lambda = \roi(\Sigma)$. A second application of the same lemma to the second and third equations, over the rigid space $\Sp(\TT_{\Sigma,h})$, now shows that, after potentially shrinking $\Sigma$ and $V$ again, we have $\hc(\Y,\DD_{\Sigma})\ssh\otimes_{\TT_{\Sigma,h}}T$ free of rank one over $T$, as required.
\end{proof}

%
%
\section{The parallel weight eigenvariety}\label{sec:parallel weight ev}
We describe a `parallel weight eigenvariety' $\Epar \subset \E$, using $\hc$ over the parallel weight line, that contains the base-change image and is better behaved than the whole space $\E$. This bears comparison with the `middle-degree eigenvariety' of \cite{BH17}. We show smoothness at certain classical points in the base-change image.

\subsection{Definition and basic properties}
In \cite{Han17}, the eigenvariety $\E$ arises from a datum $\mathfrak{D} = (\W_K,\mathscr{L},\mathscr{M},\HH_{\n,p},\psi)$, where $\mathscr{L}$ is a Fredholm hypersurface, $\mathscr{M}$ is a coherent sheaf on $\mathscr{L}$ given by (total) overconvergent cohomology, and $\psi : \HH_{\n,p} \rightarrow \mathrm{End}_{\roi(\mathscr{L})}(\mathscr{M})$ is the natural map. Define an eigenvariety datum $\mathfrak{D}_{\mathrm{par}} \defeq (\W_{K,\mathrm{par}},\mathscr{L}_{\mathrm{par}},\mathscr{M}_{\mathrm{par}}^1,\HH_{\n,p},\psi_{\mathrm{par}}),$
where:
\begin{itemize}\setlength{\itemsep}{1pt}
\item[(i)] $\W_{K,\mathrm{par}}$ is the parallel weight line in $\W_K$, i.e.\ the image of $\W_{\Q}$ in $\W_K$;
\item[(ii)] $\mathscr{L}^{\mathrm{par}}$ is the union of the irreducible components of $\mathscr{L}$ that lie above $\W_{K,\mathrm{par}}$, which is itself a Fredholm hypersurface;
\item[(iii)] $\mathscr{M}_{\mathrm{par}}^1$ is the coherent sheaf on $\mathscr{L}^{\mathrm{par}}$ such that for any slope $\leq h$ affinoid $\mathscr{L}_{\Sigma,h}^{\mathrm{par}}$ lying above $\Sigma \subset \W_{K,\mathrm{par}}$, we have $\mathscr{M}^1_{\mathrm{par}}(\mathscr{L}_{\Sigma,h}^{\mathrm{par}}) = \hc(\Y,\DD_{\Sigma})\ssh$ (as in \cite[Prop.\ 4.3.1]{Han17});
\item[(iv)] and $\psi_{\mathrm{par}}:\HH_{\n,p} \rightarrow \mathrm{End}_{\roi(\mathscr{L}_{\mathrm{par}})}(\mathscr{M}_{\mathrm{par}}^1)$ is naturally induced from $\psi$.
\end{itemize}
That this datum does give a well-defined eigenvariety, denoted $\Epar$, is a simple check using the machinery developed in \cite[\S3,\S4]{Han17}. 

\begin{proposition}\label{prop:parallel weight}
Every irreducible component of $\Epar$ has dimension 1 and contains a very Zariski-dense set of classical points.  
\end{proposition}

\begin{proof}
For the dimension statement, we may check locally over some $\Sigma = \Sp(\Lambda)$ by the properties of irreducible components \cite[Lem.\ II.7.4]{Eigenbook}. As $\hc(Y_1(\n),\DD_\Sigma)\ssh$ is projective over $\Lambda$ (Lem.\ \ref{lem:torsion free}), this follows as in \cite[Thm.\ 4.5.1(i)]{Han17}.  

The classical weights correspond to classical weights in $\W_{\Q}$, so are very Zariski-dense in $\W_{K,\mathrm{par}}$. Let $(\Sigma,h)$ be a slope-adapted affinoid for $\Sigma \subset \W_{K,\mathrm{par}}$ containing a classical weight.  Note $\Sigma$ contains a very Zariski-dense set of classical weights such that $h$ is a small slope, and every point in $\Sp(\TT_{\Sigma,h})$ above these weights is classical by Thm.\ \ref{control theorem}. Thus the classical points are very Zariski-dense in $\Sp(\TT_{\Sigma,h})$ (cf.\ \cite[Prop.\ II.8.6]{Eigenbook}); and gluing, the same is true in $\Epar$.
\end{proof}

\begin{proposition}\label{prop:reduced} 
	The parallel weight eigenvariety $\Epar$ is reduced.
	\end{proposition}
	\begin{proof}
	We closely follow \cite[Thm.\ II.8.8]{Eigenbook}. We first give a Zariski-dense set of $y \in \Epar$ with reduced local rings. For a slope-adapted $(\Sigma,h)$ containing a classical weight, let $R = \Lambda$ and $M = \hc(\Y,\DD_{\Sigma})\ssh$; then $M$ is finite projective over $R$ (Lem.\ \ref{lem:torsion free}). 
	Let $Z \subset \Sigma = \Sp(\Lambda)$ be the set of classical weights $\kappa = (k,k)$ such that $h < (k+1)/2$; this set is Zariski-dense. Let $Y = w^{-1}(Z) \subset \Sp(\TT_{\Sigma,h})$. Then: 
	\begin{itemize}\setlength{\itemsep}{0pt}
		\item[(1)] each $y \in Y$ is a non-critical slope classical point, and 
		\item[(2)] $\f_y$ is either new at $\pri|p$ or a regular $p$-stabilisation (as irregular stabilisations have slope $(k+1)f_{\pri}/2$, for $f_{\pri}$ the inertia index of $\pri|p$).
\end{itemize}
 By (2), $U_{\pri}$ acts semisimply on $\hc(\Y,\VV_\kappa^*)_{(\f_y)}$ for each $\pri|p$. The operators $T_{\mathfrak{q}}$ for $\mathfrak{q}\nmid \n$ act semisimply, as they commute with their adjoints under the natural Petersson inner product \cite[(3.4a) and before (8.2a)]{Hid94}. Finally, the operators $\langle v\rangle$ act semisimply as they have finite order. We deduce $\HH_{\n,p}$ acts semisimply on $\hc(\Y,\VV_\kappa^*)_{(\f_y)}$, hence on  $\hc(\Y,\DD_\kappa)_{(\f_y)}$ by (1). Similarly to \eqref{eq:injection}, for each $\kappa \in Z$ we have a Hecke-equivariant inclusion
	\[
		M \otimes_\Lambda\Lambda/\m_\kappa \hookrightarrow \hc(\Y,\DD_\kappa)\ssh = \bigoplus\limits_{y \in Y, w(y) = \kappa} \hc(\Y,\DD_\kappa)_{(\f_y)},
	\]
	so $\HH_{\n,p}$ acts semisimply on $M\otimes_\Lambda\Lambda/\m_\kappa$ for all $\kappa \in Z$. Then \cite[Prop.\ I.9.1]{Eigenbook} implies $\TT_{\Sigma,h}$ is reduced. Ranging over all slope-adapted pairs, we obtain a Zariski-dense set of points in $\Epar$ with reduced local rings.
	
	For all $z \in \Epar$, the local ring of $\Epar$ at $z$ contains no embedded primes; this can be checked locally over a suitable slope-adapted pair $(\Sigma,h)$, whence it follows from \cite[Prop.\ I.3.4]{Eigenbook} and the projectivity of $\hc(\Y,\DD_\Sigma)\ssh$ over $\Lambda$. Reducedness of $\Epar$ now follows from \cite[Lem.\ 3.11]{Che05}.
	\end{proof}

\begin{corollary}
	There is a closed immersion $\Epar \hookrightarrow \E$.
\end{corollary}
\begin{proof}
	By reducedness and \cite[Thm.\ 3.2.1]{JoNew}, it suffices to check an inclusion of a very Zariski-dense set of points. As every classical point $x \in \Epar$ corresponds to a system of eigenvalues that appears in $\hc(\Y,\DD_\lambda)$ for some $\lambda \in \W_{K,\mathrm{par}}$, the conditions of the theorem are satisfied, giving the required closed immersion.
\end{proof}

\subsection{The base-change eigenvariety and smoothness}

By \cite[Thm. 5.1.6]{Han17}, we see that $\mathrm{BC}$ (see Rem.\ \ref{rem:zd pstab}) factors through $\CC\rightarrow\E_{\mathrm{par}}.$ Let $\E_{\mathrm{bc}} \defeq \mathrm{BC}(\CC) \subset \Epar$ denote the image.

\begin{proposition}\label{prop:smooth nc}
	Let $f \in S_{k+2}(\Gamma_1(N))$ be an eigenform satisfying Conditions \ref{running assumptions Q}, corresponding to $x_f \in \CC(L)$. Suppose $f$ is non-critical. Then $\mathrm{BC} : \CC \to \E_{\mathrm{bc}}$ is locally an isomorphism at $x_f$, and hence $\E_{\mathrm{bc}}$ is smooth at $\mathrm{BC}(x_f)$.
\end{proposition}

\begin{proof}
 We look more closely at the construction of $\mathrm{BC}$.
The Coleman--Mazur eigencurve arises from an eigenvariety datum $(\W_{\Q}, \mathscr{L}_{\Q},$ $\mathscr{M}_{\Q},$ $\bH_{\Q,N,p},$$\psi_{\Q})$. There is a natural map $\phi : \bH_{\n,p} \rightarrow \bH_{\Q,N,p}$ (see \cite[\S4.3]{JoNew}). We define a new eigenvariety datum $(\W_{\Q},\mathscr{L}_{\Q},\mathscr{M}_{\Q},\bH_{\n,p}, \psi_{\Q}\circ \phi)$, giving rise to an intermediate eigenvariety $\CC^K$. Let $\Sigma = \Sp(\Lambda)$ be an affinoid in $\W_{\Q}$ which is slope-$h$ adapted for $\mathscr{M}_{\Q}$; then there is a map $\mathrm{BC}':\CC_{\Sigma,h} \rightarrow \CC^K_{\Sigma,h}$ arising from the inclusion $\roi(\CC^K_{\Sigma,h}) \subset \roi(\CC_{\Sigma,h})$ of $\Lambda$-algebras induced by the inclusion $\phi(\HH_{\n,p}) \subset \HH_{\Q,N,p}$. By \cite[Thm.\ 5.1.2]{Han17}, there is a closed immersion $\CC^K \hookrightarrow \E_{\mathrm{par}}$, and the map $\mathrm{BC}$ is the composition $\CC \rightarrow \CC^K \hookrightarrow \E_{\mathrm{par}}$. It suffices to show that $\mathrm{BC}'$ is locally an isomorphism at $x_f$. 

Since $f$ is non-critical, after localising and base-extending $\Lambda$, by \cite{Bel12} we know that $\roi(\CC_{\Sigma,h})_{\m_x}$ is free of rank one over $\Lambda\locl$. Since $\roi(\CC^K_{\Sigma,h})_{\m_{\mathrm{BC}'(x)}}$ is a $\Lambda\locl$-subalgebra containing 1, it must be isomorphic to $\roi(\CC_{\Sigma,h})_{\m_x}$, and $\mathrm{BC}'$ is locally an isomorphism at $x_f$. As $\CC$ is smooth at $x_f$ (see \cite[Thm.\ 2.16]{Bel12}), we deduce that $\CC^K$ is smooth at $\mathrm{BC}'(x_f)$, as required.
\end{proof}

 We now consider the analogue of Prop.\ \ref{prop:smooth nc} when $f$ is critical. We need an additional mild hypothesis, following \cite[\S1.4]{Bel12}:
 
\begin{definition}\label{def:decent}
	We say $f$ is \emph{decent} if $f$ is non-critical, or $f$ has vanishing adjoint Selmer group $\h^1_f(\Q,\ad\rho_f) = 0$, where $\rho_f: G_{\Q} \rightarrow \mathrm{GL}_2(L)$ is the $p$-adic Galois representation attached to $f$.
\end{definition}

All $f$ are conjectured to be decent (see \cite[\S 2.2.4]{Bel12}). Vanishing of $\h^1_f(\Q,\ad \rho_f)$ is proved under conditions on the residual image in \cite{All16}. 

Now suppose $f$ is critical and decent. Let $x\defeq \mathrm{BC}(x_f)$ and denote by $\mathfrak{t}_x$ the tangent space of  $\E_{\mathrm{bc}}$ at $x$. As $\E_{\mathrm{bc}}$ is a curve, we know $\mathrm{dim}_L \ \mathfrak{t}_x \geq 1$; so to prove $\E_{\mathrm{bc}}$ is smooth at $x$, it suffices to show $\mathrm{dim}_L \ \mathfrak{t}_x \leq 1$. We use deformations of Galois representations, adapting \cite[Thm.\ 2.16]{Bel12}.
\begin{notation}\label{defn:restricted ramification}
Let $S$ be the union of the infinite place with the set of places of $\Q$ supporting $N$, and $S_K$ the set of places of $K$ lying over $S$. We let $ G_{\Q,S}$ and $G_{K,S_K}$ be the Galois groups of the maximal algebraic extension of $\Q$ (resp.\ $K$) ramified only at the places $S$ (resp.\ $S_K$). 
\end{notation}
Note $\rho_f$ factors through $G_{\Q, S}$; from now on we consider $\rho_f$ as defined on $G_{\Q, S}$.  Let $\rho_x = \rho_{f}|_{G_{K, S_K}}$, the Galois representation attached to $x$. Throughout, we use decomposition groups 
\begin{equation}
\label{eq: decomp groups}
G_{K_\mathfrak{q}} \to G_{K, S_K}, \qquad G_{\Q_q} \to G_{\Q,S}
\end{equation}
and complex conjugation $c \in G_{\Q,S}$ defined by the choices of embeddings from \S\ref{sec: 2}. Likewise, $I_\mathfrak{q} \subset G_{K_\mathfrak{q}}$ denotes an inertia subgroup; similarly, we use $I_q$ over $\Q$. 

\begin{definition}\label{def:deformation problem proof critical case} Let $\mathcal{A}_L$ denote the category of Artinian local $L$-algebras $A$ with residue field $L$, and for each $A \in \mathcal{A}_L$, let $X^{\mathrm{ref}}(A)$ be the set of deformations (under strict equivalence) $\rho_A$ of $\rho_x$ to $A$ satisfying the following.

\begin{itemize}\setlength{\itemsep}{1pt}
\item[(i)]  If $\mathfrak{q}$ is a prime of $K$ dividing $\n$ but coprime to $p$, then $\rho_A|_{I_{\mathfrak{q}}}$ is constant. 
	\item[(ii)] For each $\mathfrak{p}\mid p$ in $K$, we have:
	\begin{itemize} \setlength{\itemsep}{1pt}
		\item[(1)] (\emph{null weights}) for each embedding $\tau : K_{\pri} \hookrightarrow L$, one of the $\tau$-Hodge--Sen--Tate weights of $\rho_A|_{G_{K_{\pri}}}$ is $0$;
	
		\item[(2)]\label{condition:crystalline} (\emph{crystalline periods/weakly refined}) there exists $\widetilde{\alpha}_{\mathfrak{p}} \in A$ such that the $K_{\mathfrak{p}}\otimes_{\Q_p} A$-module $D_{\mathrm{crys}}(\rho_A|_{G_{K_{\mathfrak{p}}}})^{\varphi^{f_{\mathfrak{p}}}=   \widetilde{\alpha}_{\mathfrak{p}}}$ is free of rank $1$ and $(\widetilde{\alpha}_{\pri} \newmod{\m_A})$$ = \alpha_{\pri}$, where $f_{\mathfrak{p}}$ is the inertia degree of $\mathfrak{p}$.
	\end{itemize} 
\end{itemize}

Let $X^{\mathrm{ref},\mathrm{bc}}(A)$ to be the set of deformations $\rho_A \in X^{\mathrm{ref}}(A)$ also satisfying:
\begin{itemize}
\item[(iii)] (\emph{base-change}) $\rho_A$ admits an extension to $G_{\Q, S}$ deforming $\rho_f$.
\end{itemize}
Write $\mathfrak{t}^{\mathrm{ref}} \defeq X^{\mathrm{ref}}(L[\varepsilon])$ and $\mathfrak{t}^{\mathrm{ref},\mathrm{bc}} \defeq X^{\mathrm{ref},\mathrm{bc}}(L[\varepsilon])$ for the corresponding tangent spaces where, as usual, $L[\varepsilon] = L[X]/(X^2)$.
\end{definition}

We can evaluate $\rho_f$ at complex conjugation $c$, and note that the operation
\[
\iota: \rho_A \longmapsto \big[\ad\rho_f(c) \cdot \rho_A^c : g \mapsto \rho_f(c)\rho_A(cgc)\rho_f(c)\big]
 \]
 is a functorial involution on $X^{\mathrm{ref}}$. We thank Carl Wang-Erickson for explaining the utility of this involution, and for supplying the appendix that proves the following.
 
 \begin{proposition}\label{prop:appendix}
 \begin{itemize} \setlength{\itemsep}{1pt}
 \item[(i)] The fixed point functor $(X^{\mathrm{ref}})^\iota$ is canonically isomorphic to $X^{\mathrm{ref},\mathrm{bc}}$.
 \item[(ii)] The deformation problems $X^{\Refrm,\bc}, X^{\Refrm}$ on $\mathcal{A}_L$ are pro-represented by complete Noetherian local rings $R^{\Refrm,\bc}, R^{\Refrm} \in \mathcal{A}_L$. The involution $\iota$ induces an automorphism $\iota^* : R^{\Refrm} \to R^{\Refrm}$, and there is a natural surjection
 \[
  R^{\Refrm} \twoheadrightarrow \frac{R^{\Refrm}}{((1 - \iota^*)(R^{\Refrm}))} \cong R^{\Refrm,\bc}.
 \]
 \item[(iii)] There is a canonical injection $\mathfrak{t}^{\Refrm,\bc} \hookrightarrow \mathfrak{t}^{\Refrm}$ of tangent spaces, whose image is the subspace $(\mathfrak{t}^{\Refrm})^\iota$ fixed by the involution $\iota_* : \mathfrak{t}^{\Refrm} \to \mathfrak{t}^{\Refrm}$ induced by $\iota$.
 \end{itemize}
 \end{proposition}

\begin{lemma}\label{lem:galois family}
There exists a neighbourhood $V$ of $x$ in $\E_{\mathrm{bc}}$ and a Galois representation $\rho_{V}: G_{K, S_K} \rightarrow \mathrm{GL}_2(\roi(V))$ such that for each classical point $y \in V$, the specialisation $\rho_{V, y}$ of $\rho_V$ at $y$ is the Galois representation attached to $y$.
\end{lemma}
\begin{proof}
By a theorem of Rouquier and Nyssen (see \cite{Rou96} or \cite{Nys96}), one obtains such a representation from the Galois pseudorepresentation on $\E_{\mathrm{bc}} \subset \E$ constructed in \cite{JoNewExt}. One can check that if $V_{\Q}$ is a suitable neighbourhood of $x_f$ in $\CC$, then $\rho_V$ is the restriction of $\rho_{V_{\Q}}$ to $G_{K, S_K}$, where $\rho_{V_{\Q}} : G_{\Q, S} \rightarrow \GLt(\roi(V_{\Q}))$ lifts $\rho_f$. (This restriction can be seen to take values in the subring $\roi(V) \subset \roi(V_{\Q})$ by using the explicit description of this inclusion in \cite{JoNew}).
\end{proof}

\begin{proposition}\label{prop:smooth}
	Let $f \in S_{k+2}(\Gamma_1(N))$ satisfy Conditions \ref{running assumptions Q}, corresponding to $x_f \in \CC(L)$. Suppose $f$ is critical and decent. Then  $\E_{\mathrm{bc}}$ is smooth at $x = \mathrm{BC}(x_f)$.
\end{proposition}

\begin{proof}
By the discussion after Def.\ \ref{def:decent}, it suffices to prove that the tangent space $\mathfrak{t}_x$ of $\E_{\mathrm{bc}}$ at $x$ has dimension at most 1. Let $\roi_x$ be the local ring of $\E_{\mathrm{bc}}$ at $x$. After localising $\rho_V$ at $x$, we obtain a representation $\rho_{V,x}: G_{K, S_K} \rightarrow \mathrm{GL}_2(\roi_x)$.  Now if $I$ is a cofinite length ideal of $\roi_x$, then from the interpolation property of $\rho_{V}$ and \cite[Prop.\ 4.1.13]{Liu15} we deduce that $\rho_{V,x}\otimes \roi_x/I$ satisfies condition (ii,2) defining $X^{\mathrm{ref},\mathrm{bc}}$, with $\widetilde{\alpha}_{\pri}$ the image of the $U_{\pri}$ operator in $\roi_x/I$ (see also Lem.\ \ref{lem: Q to K} of the appendix). Using the same argument as in the proof of \cite[Thm.\ 2.16]{Bel12}, or using the fact that $\rho_V = \rho_{V_{\Q}}|_{G_{K, S_K}}$, we deduce conditions (i) and (ii,1). We have a given extension to $G_{\Q, S}$, giving (iii). Thus the strict class of $\rho_{V,x}\otimes \roi_x/I$ is an element of $X^{\mathrm{ref}, \mathrm{bc}}(\roi_x/I)$. Considering the universal property, and taking the limit with respect to $I$, we obtain a morphism $R^{\mathrm{ref}, \mathrm{bc}} \rightarrow \widehat{\roi}_{x}$, the target being the completed local ring at $x$. A standard argument (see \cite[Prop. 4.5]{Berg17}) shows that this morphism is surjective. It follows that $\mathrm{dim}_L \ \mathfrak{t}_x \leq \mathrm{dim}_L \ \mathfrak{t}^{\mathrm{ref}, \mathrm{bc}}$.

To bound the dimension of $\mathfrak{t}^{\mathrm{ref}, \mathrm{bc}}$, we reduce to a result of Bella\"{i}che. Indeed, in \cite[Thm.\ 2.16]{Bel12}, he defines a deformation functor $D$ on $G_{\Q}$-representations deforming $\rho_f$, satisfying the $G_{\Q}$ analogues of the conditions defining $X^{\mathrm{ref}}$. Using the hypothesis that $\h^1_f(\Q,\ad\rho_f) = 0$, he bounds the dimension of the Zariski tangent space of $D$, which he denotes $t_D$, by 1. We will show there exists an isomorphism $t_D \cong \mathfrak{t}^{\mathrm{ref},\mathrm{bc}}$. Indeed, by ignoring all the deformation conditions, we can view $\mathfrak{t}^{\mathrm{ref},\mathrm{bc}}$ as a subspace of the tangent space without conditions, which we identify with $\h^1(K,\ad\rho_x)$. Using condition (iii) and Prop.\ \ref{prop:appendix} it is moreover a subspace of $\h^1(K,\ad\rho_x)^\iota \cong \h^1(\Q,\ad\rho_f)$.

As $\mathrm{dim}_L \ t_D \leq 1$, the result then follows from the following claim.
\end{proof}

 \begin{claim}
 Under $\phi: \h^1(K,\ad\rho_x)^\iota \isorightarrow \h^1(\Q,\ad\rho_f)$, the tangent space $\mathfrak{t}^{\mathrm{ref,bc}}$ is mapped isomorphically onto the tangent space $t_D$ from \cite[Thm.\ 2.16]{Bel12}.
 \end{claim}
\begin{proof}If $\rho_x^\varepsilon \in \mathfrak{t}^{\mathrm{ref},\mathrm{bc}}$, then it admits an extension $\rho_f^\varepsilon$ to $G_{\Q}$ deforming $\rho_f$. By Lem.\ \ref{lem: Q to K} of the appendix, $\rho_f^\varepsilon$ satisfies precisely the conditions required to be in $t_D = D(L[\varepsilon])$ in \cite{Bel12}. Hence $\phi(\mathfrak{t}^{\mathrm{ref},\mathrm{bc}}) \subset t_D$. If conversely we take a deformation $\rho_f^\varepsilon \in t_D$, then again by Lem.\ \ref{lem: Q to K} we have $\rho_f^\varepsilon|_{G_K} \in X^{\mathrm{ref}}(L[\varepsilon])$. But by definition this restriction also lies in $X^{\mathrm{ref},\mathrm{bc}}(L[\varepsilon])$, so in fact in $\mathfrak{t}^{\mathrm{ref},\mathrm{bc}}$. This is enough to show that $t_D \subset \phi(\mathfrak{t}^{\mathrm{ref},\mathrm{bc}})$, completing the proof of the claim. 
\end{proof}

\subsection{The $\Sigma$-smoothness condition}\label{sec:sigma smooth}

We would like to conclude that $\Epar$ is smooth at base-change points (or twists thereof). However, there might exist other irreducible components of $\Epar$, not contained in $\E_{\mathrm{bc}}$, that meet $\E_{\mathrm{bc}}$ at such points. 

\begin{definition}\label{def:sigma smooth}
A point $x \in \E_{\mathrm{bc}}$ is \emph{$\Sigma$-smooth} if every irreducible component $\mathcal{I} \subset \Epar$ through $x$ is contained in $\E_{\mathrm{bc}}$ (equivalently, if the natural inclusion $\Ebc \subset \Epar$ is locally an isomorphism at $x$).
\end{definition}

 For decent $f$ satisfying Conditions \ref{running assumptions Q} with base-change $\f$, by Props.\ \ref{prop:smooth nc} and \ref{prop:smooth} $\Epar$ is smooth at $x_{\f}$ if and only if $x_{\f}$ is $\Sigma$-smooth. If $\f$ is non-critical then $x_{\f}$ is $\Sigma$-smooth by Cor.\ \ref{prop:free neighbourhood}.

We conjecture that every decent classical base-change point is $\Sigma$-smooth. At non-critical points, this holds by Cor.\ \ref{prop:free neighbourhood}. In general, this is implied by the following generalisation of a conjecture of Calegari--Mazur \cite[Conj.\ 1.3]{CM09}.

Recall if $y$ is a classical point of the Bianchi (resp.\ Coleman--Mazur) eigenvariety, then $\f_y$ (resp.\ $f_y$) is the corresponding modular form. If $y$ is such a point, and $\varphi$ is a finite order Hecke character of $K$ (resp.\ $\Q$), write $y \otimes \varphi$ for the classical point (in the relevant eigenvariety) attached to the modular form $\f_y \otimes \varphi$ (resp.\ $f_y \otimes \varphi$). Note $y\otimes \varphi$ might appear in an eigenvariety of different tame level to $y$.

\begin{conjecture}\label{conj:cal-maz precise}
	Let $\mathcal{I}$ be a non-ordinary irreducible component of $\E_{\mathrm{par}}$. There exists an integer $M$ prime to $p$, an irreducible component $\mathcal{J}$ of the Coleman--Mazur eigencurve $\CC_M$ of tame level $\Gamma_1(M)$, and a finite order Hecke character $\varphi$ of $K$, such that $\mathcal{I} = \mathrm{BC}(\mathcal{J}) \otimes \varphi$ in the following sense: for all classical points $y$ of $\mathcal{I}$, there exists a classical $z \in \mathcal{J}$ such that $y = \mathrm{BC}(z) \otimes \varphi$.
\end{conjecture}

	Calegari and Mazur conjecture that every \emph{ordinary} component of $\Epar$ is either twisted base-change (as in Conj.\  \ref{conj:cal-maz precise}) or is CM (so is transfer from a $\GL_1$-eigenvariety). Non-ordinary CM components do not exist (by slope considerations).

	This is a Bianchi version of a folklore conjecture, which says that automorphic representations vary in $p$-adic families with a Zariski-dense set of classical points if and only if they satisfy a self-duality condition \cite{APS08}, \cite[Intro.]{Urb11}.

\begin{proposition}
Conj.\ \ref{conj:cal-maz precise} implies that every classical base-change point $x_{\f} = \mathrm{BC}(x_f)$ is $\Sigma$-smooth.	
\end{proposition}

	\begin{proof} 
		Let $\cI \subset \E_{\mathrm{par}}$ be an irreducible component through $x_{\f}$. We must prove that $\cI$ is contained in $\E_{\mathrm{bc}}$.  If $\cI$ is ordinary, then $x_{\f}$ is ordinary and hence small slope; so $\E_{\mathrm{par}}$ is \'etale over $\Sigma$ by Thm.\ \ref{thm:free rank 1}, hence smooth, and $x_{\f}$ is $\Sigma$-smooth.
		
		 Suppose $\cI$ is non-ordinary. By the conjecture, there exists some $M$, an irreducible component $\cJ \subset \CC_M$ and a Hecke character $\varphi$ of $K$ such that $\cI = \mathrm{BC}(\cJ) \otimes \varphi.$  Thus there exists some classical modular form $g$ such that $\mathrm{BC}(x_g) \otimes \varphi = \mathrm{BC}(x_f)$, and we have an equality of Galois representations $\rho_f|_{G_K} = \rho_g|_{G_K}\otimes \varphi$, identifying $\varphi$ with its associated Galois character (via class field theory). Since $\rho_f|_{G_K}$ and $\rho_g|_{G_K}$ both admit extensions to $G_{\Q}$, so does $\varphi$, and it follows that there exists a Dirichlet character $\varphi_{\Q}$ such that $\varphi = \varphi_{\Q} \circ N_{K/\Q}$. Further, perhaps after multiplying $\varphi_{\Q}$ by the quadratic character $\chi_{K/\Q}$ attached to $K/\Q$, we see that $\rho_f =  \rho_g \otimes \varphi_{\Q}$. 
		
		Let $y \in \cI$ be a classical point. By the conjecture, $y = \mathrm{BC}(z) \otimes \varphi$ for some classical $z \in \cJ$. We see $\rho_{f_z}|_{G_K}\otimes \varphi = (\rho_{f_z}\otimes \varphi_{\Q})|_{G_K}$.  By the same argument as Proposition \ref{prop:zd pstab}, for $z$ in a neighbourhood of $x_g$ in $\CC_M$ we have $\rho_{f_z}\otimes \varphi_{\Q}$ unramified outside $N$, so $z \otimes \varphi_{\Q}$ appears in $\CC$. For such $z$, we have $y = \mathrm{BC}(z)\otimes\varphi = \mathrm{BC}(z\otimes\varphi_{\Q}) \in \E_{\mathrm{bc}}$. The set of such $y$ is Zariski-dense in $\cI$, as it accumulates at $x_{\f}$. Thus a Zariski-dense set of  points in $\cI$ appear in $\E_{\mathrm{bc}}$, so $\cI^{\mathrm{red}} \subset \E_{\mathrm{bc}}$ by e.g.\ \cite[Thm.\ 3.2.1]{JoNew}. But by Prop.\ \ref{prop:reduced}, $\cI = \cI^{\mathrm{red}}$ is reduced, and we conclude. 
	\end{proof}

%
%

\section{Three-variable and critical $p$-adic $L$-functions}\label{sec:three variable}
 
 Throughout \S\ref{sec:three variable}, let $f \in S_{k+2}(\Gamma_1(N))$ be a decent eigenform satisfying Conditions \ref{running assumptions Q}, and let $\f \in S_{\lambda}(U_1(\n))$ be its base-change. Then $\f$  satisfies Conditions \ref{running assumptions}, and by the previous section it varies in a family $V = \Sp T\subset \E_{\mathrm{bc}} \subset \E_{\mathrm{par}}$ over $\Sigma = \Sp\Lambda \subset \W_{K,\mathrm{par}}$. If $\f$ is critical, suppose it is $\Sigma$-smooth in the sense of Def.\ \ref{def:sigma smooth}, whence $\Epar$ is smooth at $x_{\f}$ by Prop.\ \ref{prop:smooth}. We will write $x \defeq x_{\f} \in V$ for the point corresponding to $\f$, with associated maximal ideal $\m_x \subset T$, and let $e$ be the ramification degree of $V \to \Sigma$ at $x$. We also write $y$ for a general point of $V$, with associated maximal ideal $\m_y \subset T$; if $y$ is classical we write $\f_y$ for the associated Bianchi modular form.
  
  	By $\Sigma$-smoothness, $V$ is the unique irreducible component of $\Sp(\TT_{\Sigma,h}) \subset \E_{\mathrm{par}}$  through $x$. Possibly shrinking $\Sigma$, we may take $V$ connected and smooth, so there exists an idempotent $\varepsilon$ on $\TT_{\Sigma,h}$ such that $T = \varepsilon\TT_{\Sigma,h} \subset \TT_{\Sigma,h}$ is a summand, and 
  \[
  \hc(\Y,\DD_\Sigma)\ssh\otimes_{\TT_{\Sigma,h}}T = \varepsilon\hc(\Y,\DD_\Sigma)\ssh \subset \hc(\Y,\DD_\Sigma)\ssh.
  \]
  
\subsection{Three-variable $p$-adic $L$-functions}\label{sec:subsec three variable}

Recall the Mellin transform $\mel : \hc(\Y,\DD_\Sigma)\ssh \to$ $\Dla(\cl_K(p^\infty),\Lambda)$ is valued in a space of three-variable analytic functions -- two variables coming from functions on $\cl_K(p^\infty)$, and one variable on $\Sigma$.  When $\f$ is critical, we expect that $V \to \Sigma$ is not \'etale, and we cannot identify $V$ and $\Sigma$; then the $p$-adic $L$-function should be an element of $\mathcal{D}(\cl_K(p^\infty),T)$, not $\Dla(\cl_K(p^\infty),\Lambda)$. Following Bella\"{i}che, we define a Mellin transform over $V$, rather than $\Sigma$. For this, we consider the space $\hc(\Y,\DD_\Sigma)\ssh \otimes_\Lambda T,$ which has natural $T$-structures on each factor (with $T$ acting on $\hc(\Y,\DD_\Sigma)\ssh$ via $T \subset \TT_{\Sigma,h}$). The two $T$-structures are not the same in general.

\begin{definition} 
Let $y \in V(L)$ and $\kappa = w(y)$. Define 
\begin{align*}
\mathrm{sp}_{y,2} : \hc(\Y,\DD_\Sigma)\ssh \otimes_{\Lambda} T  &\longrightarrow \hc(\Y,\DD_\Sigma)\ssh \otimes_{\Lambda} T/\m_y\\
&\subset \hc(\Y,\DD_\kappa(L))\ssh.
\end{align*}
\end{definition}

This map (`specialisation in the second factor') is equivariant for the action of the Hecke operators on the cohomology in both the target and source (i.e.\ for the first $T$-structure on the source). It is \emph{not} in general equivariant if we equip the source with the second $T$-structure and the target with the natural Hecke action.

Similarly, we can define a specialisation map at the level of distributions.
\begin{definition}
With $y$ as above, define $\mathrm{sp}_y$ to be the map
 \begin{align*}
  \mathrm{sp}_y: \Dla(\cl_K(p^\infty),\Lambda)\otimes_{\Lambda}T &\longrightarrow \Dla(\cl_K(p^\infty),\Lambda)\otimes_{\Lambda}T/\m_y\\
&= \Dla(\cl_K(p^\infty),\Lambda)\otimes_\Lambda L \cong \Dla(\cl_K(p^\infty),L), 
 \end{align*}
where the last isomorphism is \cite[Prop. 2.2.1]{Han17}.
\end{definition}

Define the \emph{Mellin transform over $V$} to be the map $\mel_V \defeq \mel\otimes\mathrm{id}$ in the top row of \eqref{eqn:ev and mellin}. From the definitions, we see the following diagram commutes:
\begin{equation}\label{eqn:ev and mellin}
 \xymatrix@C=15mm{
 	\hc(\Y,\DD_\Sigma)\ssh\otimes_{\Lambda}T \ar[r]^-{\mel_V}\ar[d]^-{\mathrm{sp}_{y,2}}    & \Dla(\cl_K(p^\infty),\Lambda)\otimes_{\Lambda}T  \ar[d]^-{\mathrm{sp}_y}\\
\hc(\Y,\DD_\kappa(L))\ssh \ar[r]^-{\mel} 	&\Dla(\cl_K(p^\infty),L).
}
\end{equation}

	Recall the $p$-adic $L$-function at $y$ is defined as the Mellin transform of a class in the generalised eigenspace $\hc(\Y,\DD_\kappa(L))_{\m_y}$. To use \eqref{eqn:ev and mellin}, we would like to find a class $\Psi \in \hc(\Y,\DD_\Sigma)^{\leq h} \otimes_\Lambda T$ such that $\mathrm{sp}_{y,2}(\Psi)$ lies in this generalised eigenspace. We have a criterion for this:
	
	\begin{lemma}\label{lem:two structures agree}
		Let $\Psi \in \hc(\Y,\DD_\Sigma)^{\leq h}\otimes_\Lambda T$ such that $(t \otimes 1 - 1 \otimes t)\cdot\Psi = 0$ for all $t \in T$ (that is, the first and second $T$-structures agree on $\Psi$). Then $\mathrm{sp}_{y,2}(\Psi) \in \hc(\Y,\DD_\kappa(L))_{\m_y}$ is a Hecke eigenclass for all $y \in V(L)$.
	\end{lemma}
	\begin{proof}
		Let $\psi_y : T \to L$ be the system of Hecke eigenvalues at $y$. By definition of $\mathrm{sp}_{y,2}$, for all $t \in T$ we have $\mathrm{sp}_{y,2}[(1\otimes t)\cdot \Psi] = \psi_y(t)\mathrm{sp}_{y,2}(\Psi)$. Since the two $T$-actions agree on $\Psi$, we have $\mathrm{sp}_{y,2}[(1\otimes t) \cdot\Psi] = \mathrm{sp}_{y,2}[(t\otimes 1)\cdot \Psi]$; and since $\mathrm{sp}_{y,2}$ is equivariant for the Hecke actions on the cohomology, we have $\mathrm{sp}_{y,2}[(t\otimes 1)\cdot\Psi] = t\cdot\mathrm{sp}_{y,2}(\Psi)$. Combining, we have $t \cdot\mathrm{sp}_{t,2}(\Psi) = \psi_y(t) \mathrm{sp}_{t,2}$, from which we conclude.
	\end{proof}

We say an affinoid $\Sigma = \Sp(\Lambda) \subset \W_{K,\mathrm{par}}$ is \emph{nice} if $\Lambda$ is a principal ideal domain. Every classical weight has a basis of nice affinoid neighbourhoods (see after \cite[Defn.\ 3.5]{Bel12}); so we may shrink $\Sigma$ and $V$ so that $\Sigma$ is nice. We now remove the assumption that $\f$ is non-critical in Thm.\ \ref{thm:free rank 1} and Cor.\ \ref{prop:free neighbourhood}.

\begin{proposition}\label{prop:critical free neighbourhood}
	\begin{itemize}
		\item[(i)] $\hc(\Y,\DD_{\Sigma})\ssh\locx$ is free of finite rank over $(\TT_{\Sigma,h})\locx$. 
		\item[(ii)] Possibly shrinking $\Sigma$, there exists a connected component $V = \Sp T\subset \Sp(\TT_{\Sigma,h})$ of $x$ such that $\hc(\Y,\DD_{\Sigma})\ssh\otimes_{\TT_{\Sigma,h}} T$ is free of rank one over $T$.
	\end{itemize}
\end{proposition}

\begin{proof}
	To prove part (i) we use \cite[Lem.\ 4.1]{Bel12}. This says that if $R$ and $T$ are discrete valuation rings, with $T$ a finite free $R$-algebra and $M$ a finitely generated $T$-module that is free as an $R$-module, then $M$ is finite free over $T$. 
	
	As $\Lambda$ is a PID, and the module $\hc(\Y,\DD_\Sigma)\ssh$ is finite over $\Lambda$ (by general properties of slope decompositions) and torsion-free (by Lem.\ \ref{lem:torsion free}), it is finite free over $\Lambda$. It follows that $\TT_{\Sigma,h}$ is also finite and torsion-free over $\Lambda$, and hence also a finite free $\Lambda$-module. Since $\W_{K,\mathrm{par}}$ and $\Epar$ are rigid curves that are smooth and reduced (Prop.\ \ref{prop:reduced}) at $\lambda$ and $x$, the local rings $\Lambda\locl$ and $(\TT_{\Sigma,h})\locx$ are discrete valuation rings. We conclude (i) by Bella\"{i}che's lemma.
	
	The argument from Cor.\ \ref{prop:free neighbourhood} shows that we can shrink $\Sigma$ and $V=\Sp T$ so $\hc(\Y,\DD_\Sigma)\ssh\otimes_{\TT_{\Sigma,h}}T$ is free of finite rank over $T$. This rank is preserved by localising at any point of $V$. Let $y\in V(L)$ be a non-critical classical cuspidal point  satisfying (\ref{eqn:mult one}); such points are Zariski-dense. By Thm.\ \ref{thm:free rank 1}, $\hc(\Y,\DD_\Sigma)\ssh_{\m_y}$ is free of rank one over $T_{\m_y}$, completing the proof.
\end{proof}

As $V$ is smooth and reduced at $x$, the extension $T_{\m_x}/\Lambda\locl$ is a finite extension of DVRs, so $T_{\m_x} = \Lambda\locl[X]/(X^e-u)$ for a uniformiser $u \in \Lambda\locl$ (recalling $e$ is the ramification degree of $w$ at $x$). Possibly shrinking, this lifts to $T = \Lambda[X]/(X^e - u)$. 

\begin{definition}\label{def:phi_V} Let $\Phi_V$ be a generator of $\hc(\Y,\DD_{\Sigma})\ssh\otimes_{\TT_{\Sigma,h}} T$ over $T$, and
	 \[
	 	\Psi_V \defeq \sum_{i=0}^{e-1}X^i\Phi_V \otimes X^{e-1-i} \in \hc(\Y,\DD_\Sigma)\ssh \otimes_\Lambda T.
	 \]
\end{definition}
Here we use $\hc(\Y,\DD_\Sigma)^{\leq h}\otimes_{\TT_{\Sigma,h}}T = \varepsilon\hc(\Y,\DD_\Sigma)\ssh \subset \hc(\Y,\DD_{\Sigma})\ssh$, for $\varepsilon$ the idempotent above. If $\f$ is non-critical, Cor.\ \ref{prop:free neighbourhood} says $e = 1$, and $\Psi_V= \Phi_V$.

\begin{lemma}\label{lem:two structures agree 2}
	For all $t \in T$, we have $(t \otimes 1 - 1 \otimes t)\cdot \Psi = 0$.
\end{lemma}
\begin{proof}
For the generator $t = X$, this sum telescopes (cf.\ \cite[Lem.\ 4.13]{Bel12}).
\end{proof}

Note $\Psi_V$ depends on the choices made only up to multiplication by an element of $T^\times$ (in either $T$-structure). We now formally have (compare \cite[Prop.\ 4.14]{Bel12}):

\begin{proposition}\label{prop:spec to y}
Let $y$ be any non-critical classical point in $V(L)$ satisfying Conditions \ref{running assumptions}.  Then $\mathrm{sp}_{y,2}(\Psi_V)$ generates the (one-dimensional) generalised eigenspace $\hc(\Y,\DD_\kappa(L))_{\m_y}$.
\end{proposition}
\begin{proof}
	This follows directly from Lems.\ \ref{lem:two structures agree} and \ref{lem:two structures agree 2}.
\end{proof}

 \begin{definition}
 	Let $\mathcal{L}_p(V) = \mel_V(\Psi_V) \in \Dla(\cl_K(p^\infty),\Lambda)\otimes_{\Lambda}T \cong \Dla(\cl_K(p^\infty),T).$
 \end{definition}

For $y$ as in Prop.\ \ref{prop:spec to y}, if $\Psi_{\f_y}$ is the overconvergent modular symbol attached to $\f_y$ in \S\ref{sec:overconvergent cohomology}, then $\mathrm{sp}_{y,2}(\Psi_V) = c_y \Psi_{\f_y}$ for some $c_y \in L^\times.$ By \eqref{eqn:ev and mellin} we see that $\mathrm{sp}_y(\mathcal{L}_p(V)) = \mathrm{Mel}(\mathrm{sp}_{y,2}(\Psi_V)) = c_y L_p(\f_y)$. When combined with Thm.\ \ref{thm:Wil17}, this proves Thm.\ \ref{intro:interpolation} from the introduction. To get the precise statement of Thm.\ \ref{intro:interpolation}:
	\begin{itemize}\setlength{\itemsep}{0pt}
		\item[--] for twists by a finite order Hecke character $\phi$, we define $\mathcal{L}_p^\phi(V) \defeq \mathrm{Mel}_V^\phi(\Psi_V)$, where $\mathrm{Mel}_V^\phi \defeq \mathrm{Mel}^\phi \otimes 1$ (for $\mathrm{Mel}^\phi$ as defined in Rem.\ \ref{rem:twisted});
		\item[--] we use the Amice transform \cite{ST01},\cite[Def.\ 5.1.5]{BH17} to identify $\Dla(\cl_K(p^\infty),T)$ with the space of rigid analytic functions $V \times \mathscr{X}(\cl_K(p^\infty)) \to L$;
		\item[--] and finally, we obtain the function on $V_{\Q} \times \mathscr{X}(\cl_K(p^\infty))$ by pulling back under $\mathrm{BC}: V_{\Q} \to V$.
	\end{itemize} 

For $\mathbf{h} \in (\Q_{\geq 0})^{\pri|p}$, recall the notion of $\mu \in \Dla(\cl_K(p^\infty),L)$ being \emph{$\mathbf{h}$-admissible} from \cite[Defs.\ 5.10,6.14,\S7.4]{Wil17}. These definitions generalise directly to $\mu \in \Dla(\cl_K(p^\infty),T)$ using the Banach algebra structure on $T$. Recall $\alpha_{\pri}$ is the $U_{\pri}$-eigenvalue of $\f$. Slopes at $p$ are locally constant over $V$, so we may shrink so that the slope is constant over $V$. Then identically to \cite{Wil17} we have:

\begin{proposition}\label{prop:Lp admissible}
Let $h_{\pri} = v_p(\alpha_{\pri})$ and $\mathbf{h} = (h_{\pri})_{\pri|p}$. Then $\mathcal{L}_p(V)$ is $\mathbf{h}$-admissible.
\end{proposition}

\begin{remark}
Our construction of $\mathcal{L}_p(V)$ required $\f$ to be $\Sigma$-smooth (which is conjecturally always the case). It seems likely that this is a \emph{necessary} condition to carry out the construction of $\mathcal{L}_p(V)$ via overconvergent cohomology. There are two key inputs: the class $\Psi_V$ and the Mellin transform $\mathrm{Mel}_V$. If $\f$ is not $\Sigma$-smooth, then over $V$ we have the Mellin transform but we have little control on the geometry of $\Epar$ at $\f$ and the structure of $\hc(\Y,\DD_\Sigma)_{\m_x}$, so it is hard to construct the class $\Psi_V$. One could instead work over the (unique, smooth) base-change component $V_{\mathrm{bc}} = \Sp(T_{\mathrm{bc}}) \subset V$ through $\f$. Using smoothness, one can exhibit a canonical quotient $M_{V_{\mathrm{bc}}}$ of $\hc(\Y,\DD_\Sigma)\otimes_{\TT_{\Sigma,h}}T$ that is free of rank one over $T_{\mathrm{bc}}$, and construct a `good' class $\Psi_{V_{\mathrm{bc}}}$ in this quotient. However, the Mellin transform does \emph{not} descend to this quotient. Thus $\Sigma$-smoothness appears necessary to simultaneously obtain both $\Psi_V$ and $\mathrm{Mel}_V$.
\end{remark}

\begin{remark}\label{rem:non-base-change}
	More generally, suppose we do not assume that $\f \in S_{\lambda}(U_1(\n))$ is base-change. Suppose $\f$ satisfies Conditions \ref{running assumptions} and is non-critical. By Prop.\ \ref{prop:point} and Thm.\ \ref{thm:dimension 1}, there exists at least one (not necessarily parallel) curve $\Sigma = \mathrm{Sp}(\Lambda) \subset \W_K$ such that $\f$ varies in a family $V = \mathrm{Sp}(T)$ over $\Sigma$. If $\Sigma$ is smooth at $\lambda$, then Cor.\ \ref{prop:free neighbourhood} says $V \to \Sigma$ is \'etale at $x_{\f}$ and $\hc(\Y,\DD_\Sigma)\ssh\otimes_{\TT_{\Sigma,h}} T$ is free of rank one over $\Lambda$. Let $\Psi_V$ be a generator. The methods of this section show $\mathcal{L}_p(V) \defeq \mathrm{Mel}_V(\Psi_V \otimes 1) \in \Dla(\cl_K(p^\infty),T)$ interpolates the $p$-adic $L$-functions of all classical non-critical $y \in V$. If $V$ contains a Zariski-dense set of classical points, then $\Sigma$ is parallel (as it then contains a Zariski-dense set of parallel weights; see also \cite[Thm.\ 1.1]{Serban19}), hence smooth. This construction thus gives a three-variable $p$-adic $L$-function over \emph{any} classical family through $\f$.
	
	 In this generality, classical $y$ are not always Zariski-dense in $V$, but one might expect that $\mathcal{L}_p(V)$ still carries interesting information about the arithmetic of $\f$. For example, if one could prove a functional equation for $\mathcal{L}_p(V)$, then it might be possible use $\mathcal{L}_p(V)$ to prove \cite[Conj.\ 11.2]{BW17}, giving an arithmetic description of the $\mathcal{L}$-invariants attached to $\f$ (similar to \cite{GS93,BDJ17}). 
\end{remark}

\subsection{Critical base-change $p$-adic $L$-functions}\label{sec:critical p-adic l-functions}

 Suppose now $\f$ is critical and $\Sigma$-smooth. Recall $x = x_{\f} \in V$.

\begin{definition}
Let $L_p(\f) \defeq \mathrm{sp}_x(\mathcal{L}_p(V)) \in \Dla(\cl_K(p^\infty),L)$.
\end{definition}

To prove the interpolation property for $L_p(\f)$ from Thm.\ \ref{intro:thm B} of the introduction, we give another description of $L_p(\f)$ using a strategy of Bella\"{i}che. 

\begin{theorem}\label{thm:eigenspace one dim}
	(See \cite[Cor.\ 4.8]{Bel12}). The eigenspace
	$\hc(\Y,\DD_{\lambda}(L))[\f]$ is one-dimensional over $L$, and its image under the specialisation map $\rho_\lambda$ is 0.
\end{theorem}

\begin{proof}
	The proof closely follows the strategy of Bella\"{i}che; each step is proved identically to the referenced result. Firstly, there is an isomorphism
	\[
	(\TT_{\Sigma,h})\locx\otimes_{\Lambda\locl}\Lambda\locl/\m_{\lambda} \cong (\TT_{\lambda,h})\locx,
	\]
	where $\TT_{\lambda,h}$ is the image of $\HH_{\n,p} \otimes L$ in $\mathrm{End}_L(\hc(\Y,\DD_\lambda)\ssh)$ and we again write $\m_x$ for the maximal ideal of this space corresponding to $x$ (see \cite[Cor.\ 4.4]{Bel12}).
	
	After possibly enlarging $L$, there exists a uniformiser $u$ of $\Lambda\locl$ and an isomorphism of $\Lambda\locl$-algebras 
	\[
	\Lambda\locl[X]/(X^e-u)\cong (\TT_{\Sigma,h})\locx
	\]
	that sends $X$ to a uniformiser of $(\TT_{\Sigma,h})\locx$, where $e$ is the ramification index of the weight map $w : \Epar \rightarrow \W_{K,\mathrm{par}}$ at $x$
	(see \cite[Prop.\ 4.6]{Bel12}). 
	
	This is enough to show that the generalised eigenspace $\hc(\Y,\DD_\lambda(L))_{(\f)} = \hc(\Y,\DD_\lambda(L))\locx$ has dimension $e$ over $L$ and is free of rank one over the algebra 
	\begin{equation}\label{eqn:TT isom poly}
		(\TT_{\lambda,h})\locx  \cong L[X]/(X^e).
	\end{equation}
	(See \cite[Thm.\ 4.7]{Bel12}). Under this isomorphism, we see that the eigenspace is exactly the subspace $X^{e-1}\hc(\Y,\DD_\lambda(L))_{(\f)}.$ If $e > 1$, then $X$ acts nilpotently on this eigenspace, hence its image under $\rho_\lambda$, which is thus 0.  If $e = 1$ and $\rho_{\lambda} \neq 0$, then $\rho_\lambda$ is a non-trivial map between $1$-dimensional vector spaces, so must be an isomorphism, contradicting $\f$ being critical; so $\rho_\lambda = 0$.
\end{proof}

\begin{theorem} \label{prop:critical}
	The $p$-adic $L$-function $L_p(\f)$ is $(v_{p}(\alpha_{\pri}))_{\pri|p}$-admissible. For Hecke characters $\varphi$ of $K$ of conductor $\ff|(p^\infty)$ and infinity type $0 \leq (q,r) \leq (k,k),$ we have
	\begin{equation}\label{eq:critical interpolation}
	L_p(\f,\varphi_{p-\mathrm{fin}}) = 0.
	\end{equation}
\end{theorem}
\begin{proof} 
	Firstly, $L_p(\f)$ is admissible by specialising Prop.\ \ref{prop:Lp admissible} at $x$.
	
	By Thm.\ \ref{thm:eigenspace one dim} the eigenspace  $\hc(\Y,\DD_\lambda)[\f]$ is one-dimensional; let $\Psi_{\f}$ be a generator. Now $\mathrm{Mel}(\Psi_{\f})(\varphi_{p-\mathrm{fin}}) = 0$ for all $\varphi$ as in \eqref{eq:critical interpolation}, as evaluation of $\mathrm{Mel}(\Psi_{\f})$ at $\varphi_{p-\mathrm{fin}}$ depends only on $\rho_\lambda(\Psi_{\f})$ \cite[\S7.6]{Wil17}, and $\rho_\lambda(\Psi_{\f}) = 0$. Thus it suffices to prove that $L_p(\f)$ and $\mel(\Psi_{\f})$ are equal up to rescaling. Recall $ \hc(\Y,\DD_\Sigma)\ssh\otimes_{\TT_{\Sigma,h}}T$ is free of rank one over $T$. Let $\Psi_V$ be as in Def.\ \ref{def:phi_V}. Identically\footnote{There is a typo \emph{op.\ cit.}; from the context, $\symb_\Gamma^\pm(\mathbf{D}_{k'})_{(f'_{\beta'})}$ should be $\symb_\Gamma^\pm(\mathbf{D}_{k'})[f'_{\beta'}]$.} to \cite[Prop.\ 4.14]{Bel12}, at $x$ (up to rescaling) we have $\mathrm{sp}_{x,2}(\Psi_V) = \Psi_{\f}$. Finally specialisation is compatible with Mellin transforms by \eqref{eqn:ev and mellin}, so
	\[
	L_p(\f) = \mathrm{sp}_x(\mel_V(\Psi_V)) = \mel(\mathrm{sp}_{x,2}(\Psi_V)) = \mel(\Psi_{\f}).\qedhere
	\]
\end{proof}

	By construction, $\mathcal{L}_p(V)$ is well-defined only up to the choice of $\Phi_V$ in Def.\ \ref{def:phi_V}, corresponding to changing the $p$-adic periods $\{c_y\}$. Specialising, we see $L_p(\f)$ is only well-defined up to scalar multiple. However, this scalar indeterminacy is expected, arising from scaling the periods of $\f$.
	
	Unlike in the non-critical slope case, the admissibility and interpolation properties are not sufficient to determine $L_p(\f)$ uniquely in $\Dla(\cl_K(p^\infty),L)/L^\times$.  However, up to scaling $L_p(\f)$ is uniquely determined by interpolation over $V$, and does not depend further on the method of construction:

\begin{proposition}\label{prop:well-defined}
 Suppose $L_p(\f) \neq 0$. Let $\mathcal{L}_p' : V \times \mathscr{X}(\cl_K(p^\infty)) \to L$ be analytic and $(v_{p}(\alpha_{\pri}))_{\pri|p}$-admissible satisfying the interpolation property of Thm.\ A, with possibly different constants $c_y'$. Then $\mathcal{L}_p'(x) = c\cdot L_p(\f)$ for some $c \in L$.
\end{proposition}
\begin{proof}
Let $L_p'(\f) \defeq \mathcal{L}_p'(x).$ If $L_p'(\f) = 0$ we take $c = 0$; so assume $L_p'(\f) \neq 0$. Let
\[
 C(y,\phi) \defeq \frac{\mathcal{L}_p'(V)(y,\phi)}{\mathcal{L}_p(V)(y,\phi)} \in \mathrm{Frac}\bigg(\roi\big(V\times\mathscr{X}(\cl_K(p^\infty)\big)\bigg).
\]
To see this is well-defined, note there exists a Zariski-dense set $V_{\mathrm{ss}} \subset V$ of classical small slope points $y$ in $V$ of weight $(k,k)$, where $k > 2$. For any finite order Hecke character $\varphi$ of conductor $p^r > 1$, the quantity $L(\f_y,\varphi,k+1)$ converges absolutely to a non-zero number; it follows that $L_p(\f_y,(\varphi|\cdot|^{k})_{p-\mathrm{fin}}) \neq 0$, since the $p$-adic $L$-function does not have an exceptional zero there. As every connected component of $\roi(\mathscr{X}(\cl_K(p^\infty)))$ contains a character of the form $(\varphi|\cdot|^{k})_{p-\mathrm{fin}}$, it follows that $L_p(\f_y)$ is not a zero-divisor in $\roi(\mathscr{X}(\cl_K(p^\infty)))$ (as on every such component the only zero-divisor is 0). Now let $D \in \Dla(\cl_K(p^\infty),T)$ such that $D\mathcal{L}_p(V)  = 0$. At any $y \in V_{\mathrm{ss}}$, we have $\mathrm{sp}_y(D)\mathrm{sp}_y(\mathcal{L}_p(V)) = 0$, and as $\mathrm{sp}_y(\mathcal{L}_p(V)) = c_y L_p(\f_y)$ is not a zero-divisor we see $\mathrm{sp}_y(D) = 0$. As $D$ vanishes at a Zariski-dense set of points, we have $D = 0$, so $\mathcal{L}_p(V)$ is not a zero-divisor and $C(y,\phi)$ is well-defined.

At each classical $y \in V_{\mathrm{ss}}$, the $p$-adic $L$-function $L_p(\f_y)$ is uniquely determined up to scalar multiple by its interpolation and admissibility properties, as in \cite[Thm.\ 7.4]{Wil17}. Thus $\mathcal{L}_p(y)$ and $\mathcal{L}_p'(y)$ are scalar multiples, so $C(y,-) \in L$. As such points are Zariski-dense, for \emph{all} $z \in V$ we see $C(z,-)$ is constant, that is, $C \in \mathrm{Frac}(\roi(V))$. Since (by assumption) neither $L_p(\f)$ nor $L_p'(\f)$ is zero, $C$ does not have a zero or pole at $x$. Hence we may shrink $V$ further so that $C$ has no zeros or poles, that is, $C \in \roi(V)^\times$. Specialising to $x$, we see that $L_p(\f)$ and $L_p'(\f)$ differ by scalar multiplication by $c \defeq C(x) \in L^\times$, as required.
\end{proof}

\begin{remark}
	This proof shows there is a subspace $W_{\f} \subset \Dla(\cl_K(p^\infty),L)$, of dimension $\leq 1$, such that for any $\mathcal{L}_p : V \times \mathscr{X}(\cl_K(p^\infty)) \to L$ satisfying the interpolation property of Thm.\ \ref{intro:interpolation}, we have $\mathcal{L}_p(x) \in W_{\f}$. Further, up to shrinking $V$ any two such functions $\mathcal{L}_p$, $\mathcal{L}_p'$ with $\mathcal{L}_p(x), \mathcal{L}_p'(x) \neq 0$ differ by  $\roi(V)^\times$, as claimed in the introduction (also cf.\ Prop.\ \ref{prop:factor family}).  We can also treat twisted $p$-adic $L$-functions $\mathcal{L}_p^\phi$ by using twisted Mellin transforms $\mathrm{Mel}^\phi$ (Rem.\ \ref{rem:twisted}).
\end{remark}

%
%
\section{Factorisation of base-change $p$-adic $L$-functions}\label{factorisation}
Let $f$ be a classical eigenform with base-change $\f$ to $K$, and let $\chi_{K/\Q}$ denote the quadratic Hecke character attached to $K/\Q$. Artin formalism says that for any rational Hecke character $\varphi$, we have
\[
\Lambda(\f,\varphi\circ N_{K/\Q}) = \Lambda(f,\varphi)\Lambda(f,\varphi\chi_{K/\Q}).
\]
We now prove an analogue of this for $p$-adic $L$-functions (see Thm.\ \ref{thm:factorisation}).

\subsection{$p$-adic $L$-functions attached to classical eigenforms}
Let $f \in S_{k+2}(\Gamma_1(N))$ be a decent eigenform satisfying Conditions \ref{running assumptions Q}, and let $\Lambda(f,\varphi)$ be its $L$-function, normalised to include the Euler factors at infinity. Here $\varphi$ ranges over Hecke characters of $\Q$. Denote the eigenvalue of $f$ at $p$ by $\alpha_p(f)$ and the periods of $f$ by $\Omega_{f}^\pm \in \C^\times$, which are well-defined up to $\overline{\Q}^\times$. Let $h \defeq v_p(\alpha_p(f))$. For any Dirichlet character $\chi$ of conductor $M$, let $\tau(\chi) \defeq \sum_{a \newmod{M}}\chi(a)e^{2 \pi i a/M}$ be its Gauss sum. Let $\eta = \eta_p\eta^p$ be a Dirichlet character, where $\eta_p$ has conductor $p^t$ and $\eta^p$ has conductor $C$ prime to $p$. The following is due to many people.

\begin{theorem}\label{thm:Q}
There exists a locally analytic distribution $L_p^\eta(f)$ on $\Zp^\times$ such that, for any Hecke character $\varphi = \chi|\cdot|^j$, where $\chi$ is finite order of conductor $p^n > p^t$ and $0 \leq j \leq k$, we have
\[L_p^\eta(f,\varphi_{p-\mathrm{fin}}) = \left\{\begin{array}{cl}\frac{(Cp^n)^{j+1}}{\tau((\chi\eta)^{-1})\Omega_{f}^\pm\alpha_p(f)^n}\Lambda(f,\varphi\eta) &: f \text{ is non-critical},\\
0 &: f\text{ is critical}.
\end{array}\right.\]
 The sign of $\Omega_f^\pm$ is given by $\chi\eta(-1)(-1)^j = \pm 1$. The distribution is admissible of order $h$, and if $h < k+1$, it is uniquely determined by this interpolation property.
\end{theorem}
If $\eta$ is the trivial character, we just write $L_p(f)$ for this distribution.
\begin{proof}
When $\eta$ is trivial, see e.g.\ \cite{PS11}, \cite{Bel12}; in our normalisations, $\Lambda(f,\varphi) = \Lambda(f,\overline{\chi},j+1)$. In general, we can use a slight variation of their methods. In both papers, one constructs an overconvergent modular symbol $\Psi_f$, then sets $L_p(f) \defeq \Psi_f\{0-\infty\}|_{\Zp^\times}$. If $a \in (\Z/C\Z)^\times$, then one defines a distribution $L_p^a(f)$ on $\Zp^\times$ by 
\[
L_p^a(f) \defeq [\Psi_f|\smallmatrd{1}{a}{0}{C}]\{0-\infty\}|_{\Zp^\times},
\]
 then defines $L_p^{\eta^p}(f) = \sum_{a\in(\Z/C\Z)^\times}\eta^p(a)L_p^a(f)$. Finally, we define $L_p^{\eta}(f,\varphi_{p-\mathrm{fin}})$ $\defeq L_p^{\eta^p}(f,\varphi_{p-\mathrm{fin}}\eta_p)$. Proving the interpolation result is then a formal calculation (compare Rem.\ \ref{rem:twisted}), noting that $\chi\eta = (\chi\eta_p)\eta^p$ has conductor $Cp^n$ as $p^n > p^t$.
\end{proof}

 One also has a more involved interpolation formula at $\chi$ of conductor $p^n \leq p^t$.

There are also many constructions of this in families; see for example \cite{Bel12}.
\begin{theorem}\label{thm:classical families}
Let $x_f$ be the point of $\CC$ corresponding to $f$. There exists an affinoid neighbourhood $V_{\Q}$ of $x_f$ and a locally analytic distribution $\mathcal{L}_p^{\eta}(V_{\Q}) \in \Dla(\Zp^\times,\roi(V_{\Q}))$
 such that at any classical point $y \in V_{\Q}$,
\[\mathcal{L}_p^{\eta}(y,\phi) \defeq \mathrm{sp}_{y}(\mathcal{L}_p^{\eta}(V_{\Q}))(\phi) = c_y^\pm L_p^{\eta}(f_y,\phi),\]
where $\phi$ is locally analytic on $\Zp^\times$, $c_y^\pm \in \overline{\Q}_p^\times$ depends only on $y$, and $\eta\phi(-1) = \pm 1$.
\end{theorem}

	\subsection{Statement of $p$-adic Artin formalism}
	\label{sec:base-change evs}
	
	\subsubsection{Indeterminacy of $p$-adic $L$-functions}
	
	There are indeterminacies in all the constructions of $p$-adic $L$-functions discussed so far in this paper, arising from the choices of (complex and $p$-adic) periods. We make this more precise.
	
	For a Bianchi modular form $\f$ of weight $\lambda$, the $p$-adic $L$-function $L_p(\f)$ (where it exists) is the Mellin transform of a class $\Psi_{\f} \in \hc(\Y,\DD_\lambda)$ (see Def.\ \ref{def:non-critical Lp} or Thm.\ \ref{prop:critical}).
	In general, the class $\Psi_{\f}$ is only canonical up to $p$-adic scalar, and rescaling $\Psi_{\f}$ similarly rescales $L_p(\f)$. We see that the subspace $\overline{\Q}_p \cdot L_p(\f) \subset \Dla(\cl_K(p^\infty),\overline{\Q}_p)$ is canonically defined by our construction. Note here that there is only one complex period $\Omega_{\f}$ to choose (as $\GL_2(\C)$ is connected).

	For a classical modular form $f$, as $\GL_2(\R)$ has two connected components, there are \emph{two} choices of complex period $\Omega_f^\pm$. Correspondingly, the overconvergent modular symbol $\Psi_f$ in the proof of Thm.\ \ref{thm:Q} has the form $\Psi_f = \Psi_f^+ + \Psi_f^-$, where $\Psi_f^\pm$ is in the $\pm$-eigenspace for the action of $\smallmatrd{-1}{}{}{1}$ on overconvergent modular symbols. In this case, each $\Psi_f^\pm$ can be scaled independently by $\overline{\Q}_p$.
	
Where $p$-adic $L$-functions exist, we then have $L_p^\eta(f) = L_p^{\eta,+}(f) + L_p^{\eta,-}(f)$, where $L_p^{\eta,\pm}(f)$ is defined analogously using $\Psi_f^\pm$.  These are supported on disjoint $\pm$-halves of $\cA(\Zp^\times,\overline{\Q}_p)$; precisely, if $\phi$ is a character of $\Zp^\times$, then $L_p^{\eta,\pm}(f,\phi) \neq 0$ only if $\eta\phi(-1) = \pm 1$.  The $L_p^{\eta,\pm}(f)$, like the $\Psi_f^\pm$, can be scaled independently by $\overline{\Q}_p$, but as above, we get two canonical subspaces $\overline{\Q}_p \cdot L_p^{\eta,\pm}(f), \subset \Dla(\Zp^\times,\overline{\Q}_p)$.

For $p$-adic Artin formalism, we need to consider twists of $f$ by $\chi_{K/\Q}$. Attached to $f$ and $\chi_{K/\Q}$, we have \emph{four} canonical spaces 
	\[
		\overline{\Q}_p\cdot L_p^{\pm}(f),\   \overline{\Q}_p \cdot L_p^{\chi_{K/\Q},\pm}(f) \ \subset \Dla(\Zp^\times,\overline{\Q}_p),
	\]
	 corresponding to the Mellin transforms of the (two) lines $\overline{\Q}_p\cdot\Psi_f^\pm$ of overconvergent modular symbols. Taking a product reduces this to only \emph{one} canonical space:
	\begin{lemma}\label{lem:product line}
		Independently rescaling $\Psi_f^\pm$ does not change the space
		\[
			W_f^K \defeq \overline{\Q}_p \cdot L_p(f)L_p^{\chi_{F/\Q}}(f) \subset \Dla(\Zp^\times,\overline{\Q}_p).
		\]
	\end{lemma}
	\begin{proof}
		As $K$ is imaginary, we have $\chi_{K/\Q}(-1) = -1$. Write $L_p(f) = L_p^+(f) + L_p^-(f)$ and $L_p^{\chi_{K/\Q}}(f) = L_p^{\chi_{K/\Q},+}(f) + L_p^{\chi_{K/\Q},-}(f)$. As $L_p^+(f)$ and $L_p^{\chi_{K/\Q},+}(f)$ have disjoint support, the $++$ term in the product vanishes (and similarly for the $--$ term), so
		\[
		L_p(f)L_p^{\chi_{K/\Q}}(f) = L_p^+(f)L_p^{\chi_{K/\Q},-}(f) + L_p^-(f)L_p^{\chi_{K/\Q},+}(f).
		\]
		Now, if we rescale $\Psi_f^+$ by $c^+ \in \overline{\Q}_p$ in the construction of $L_p^+(f)$ and $L_p^{\chi_{K/\Q},+}(f)$, then we scale both terms in the right-hand side by $c^+$ and remain in $W_f^K$; and similarly for rescaling $\Psi_f^-$.
	\end{proof}


\subsubsection{Statement of $p$-adic Artin formalism} 
Let $f \in S_{k+2}(\Gamma_1(N))$ be a decent eigenform satisfying Conditions \ref{running assumptions Q}, and let $\f$ be its base-change. Assume $\f$ is $\Sigma$-smooth, and let $L_p(\f) \in \Dla(\cl_K(p^\infty),\overline{\Q}_p)$ be its $p$-adic $L$-function from \S\ref{sec:critical p-adic l-functions}.

\begin{definition}
	We define the \emph{restriction of $L_p(\f)$ to the cyclotomic line}, denoted by $L_p^{\mathrm{cyc}}(\f) \in \Dla(\Zp^\times,\overline{\Q}_p)$, to be the locally analytic distribution on $\Zp^\times$ given by
	\[
	L_p^\cyc(\f,\phi) \defeq L_p(\f,\phi\circ N\FQ),
	\]
	where $\phi$ is any locally analytic function on $\Zp^\times \cong \cl_{\Q}^+(p^\infty)$.
\end{definition}

Recall $W_f^K \subset \Dla(\Zp^\times,\overline{\Q}_p)$ from Lem.\ \ref{lem:product line}, and define
\[
W_{\f}^\cyc = \overline{\Q}_p \cdot L_p^\cyc(\f) \ \subset  \Dla(\Zp^\times, \overline{\Q}_p).
\]

\begin{theorem}\label{thm:factorisation}
Suppose that $L_p^\cyc(\f)$ and $L_p(f)L_p^{\chi_{K/\Q}}(f)$ are both non-zero. Then $W_{\f}^\cyc = W_f^K$ as lines in $\Dla(\Zp^\times,\overline{\Q}_p)$. 
\end{theorem}

 We will show that it is possible to choose generators of $W_f^K$ and $W_{\f}^\cyc$ that have the same growth and interpolation properties. When the slope is $< (k+1)/2$, these properties uniquely determine the line and force the equality $W_{f} = W_{\f}^\cyc$, as we show in \S\ref{sec:very small slope}. Cases of more general slopes arise very naturally in arithmetic, however, and in that case there are \emph{a priori} an infinite number of lines $W \subset \Dla(\Zp^\times,L)$ with generators satisfying the same growth and interpolation properties. In this case, it is far from obvious that $W_f^K = W_{\f}^\cyc$, and our methods -- where we deform these lines in $p$-adic families -- are required.

	We will prove this by fixing our choices of the periods $\Omega_{\f}$ and $\Omega_{f}^\pm$ such that 
	\begin{equation}\label{eq:complex periods}
		\Omega_{\f} = (-1)^k\unitsize\Omega_f^+\Omega_f^-/2\tau(\chi_{K/\Q}),
	\end{equation}
	which pins down choices of generators $L_p^{\cyc}(\f) \in W_{\f}^\cyc$ and $L_p(f)L_p^{\chi_{K/\Q}}(f) \in W_f^K$.  This choice is possible by the parity of $\chi_{K/\Q}$ and classical Artin formalism.  We then prove that for this choice, there is an equality (of distributions on $\Zp^\times$)
	\[
	L_p^\cyc(\f) = L_p(f)L_p^{\chi_{K/\Q}}(f).
	\]

\begin{remark}
 	Evidently it would be desirable to prove a more precise version of this theorem and give an equality of elements within this line for some \emph{canonical} periods. This appears to be extremely subtle; controlling the periods, with no recourse to $p$-adic $L$-functions, is already difficult, and is the subject of \cite{TU18}. 	
\end{remark}

\subsection{Proof of Thm.\ \ref{thm:factorisation} for slope $h < (k+1)/2$}\label{sec:very small slope}
Recall the Hecke eigenvalues of $\f$ can be described simply in terms of the eigenvalues of $f$:
\begin{itemize}\setlength{\itemsep}{1pt}
	\item[(i)] When $p$ splits as $\pri\pribar$ in $K$, we have $\alpha_{\pri}(\f) = \alpha_{\pribar}(\f) = \alpha_p(f).$
	\item[(ii)] When $p$ is inert in $K$, we have $\alpha_{p\roi_K}(\f) = \alpha_p(f)^2$.
	\item[(iii)] when $p$ is ramified as $\pri^2$ in $K$, we have $\alpha_{\pri}(\f) = \alpha_p(f)$.
\end{itemize} 
 We see that $\f$ has small slope if and only if
\[
v_p(\alpha_p(f)) < \left\{\begin{array}{ll}k+1 &: p\text{ split},\\
\frac{k+1}{2} &: p\text{ inert or ramified.}
\end{array}\right.
\]
We see $\alpha_{p\roi_K}(\f) = \alpha_p(f)^2$, where $U_{p\roi_K}\f = \alpha_{p\roi_K}(\f)\f$. Thus if $h = v_p(\alpha_p(f))$, then $L_p^{\cyc}(\f)$ and $L_p(f)L_p^{\chi_{K/\Q}}(f)$ are admissible of order $2h$ \cite[Def.\ 6.1]{PS11}.

Suppose now $h < (k+1)/2$. As $2h < k+1$, both $L_p^{\cyc}(\f)$ and $L_p(f)L_p^{\chi_{K/\Q}}(f)$ are uniquely determined by their values at critical characters (e.g.\ \cite{AV75}). In this case, it thus suffices to prove the interpolation properties agree. By classical Artin formalism, the classical $L$-values at critical characters agree, so it suffices to check that the constants in the interpolation formulae agree. For characters factoring through $N_{K/\Q}$, the interpolating constant of the Bianchi $p$-adic $L$-function can be simplified. Write $d = d'p^t$, where $d'$ is prime to $p$. If $\eta$ is a Hecke character of $K$, let $\tau_K(\eta)$ be the Gauss sum from \cite[Def.\ 2.6]{BW_CJM}. If $\eta$ is finite order and $\mathrm{cond}(\eta) = M\roi_K$ is principal, then this admits an explicit  description as
	\begin{equation}\label{eqn:gauss sum}
		\tau_K(\eta) \defeq \sum_{a  \in (\roi_K/M\roi_K)^\times}\eta(a)e^{2\pi i \mathrm{Tr}_{K/\Q}\left(\frac{a}{M\sqrt{-d}}\right)},
	\end{equation}
identifying $(\roi_K/M\roi_K)^\times$ with the associated quotient of $\widehat{\roi}_K^\times$. 
\begin{remark}
	We remark that this is different from the (somewhat non-standard) Gauss sum $\widetilde{\tau}(\eta)$ from \cite[\S1.2.3]{Wil17} (which we used in Thm.\ \ref{thm:Wil17}). If $\eta$ has finite order, then $\tau_K(\eta) = \widetilde{\tau}(\eta^{-1})$ (via \cite[p.621]{Wil17}). For compatibility with \cite{Wil17}, we have tried to use $\widetilde{\tau}$ as much as possible, but for non-principal conductor it is much more convenient to use the (idelic) formulation of \cite{BW_CJM}.
\end{remark}
\begin{proposition}
Let $\phi = \chi|\cdot|^j$ with $\mathrm{cond}(\chi) = p^n > p^t$ and $0 \leq j \leq k$, and let $\varphi = \phi\circ N\FQ.$ Then
\[
L_p^\cyc(\f,\phi) = L_p(\f,\varphi) = \left[\frac{(d')^{j+1}p^{2n(j+1)}\unitsize}{(-1)^{k}2\alpha_{p\roi_K}(\f)^n\tau_K(\chi^{-1}\circ N_{K/\Q})\Omega_{\f}}\right]\Lambda(\f,\varphi).
\]
\end{proposition}
\begin{proof}
This is an exercise in book-keeping. In this setting the factors $Z_{\pri}$ of Thm.\ \ref{thm:Wil17} are equal to 1; the infinity type is $(j,j)$, simplifying the sign. If $p$ is unramified in $K$ (i.e.\ $d=d'$), then $\mathrm{cond}(\chi \circ N_{K/\Q}) = p^{n}\roi_K$, so the terms $\varphi(x_{\ff})$ and $\varphi_{\ff}(x_{\ff})$ cancel as the conductor is principal. Recall $\tau_K$ is inverse to the Gauss sum $\widetilde{\tau}$ used in \cite{Wil17} on finite order characters, whilst $\widetilde{\tau}(\phi\circ N_{K/\Q})$ and $\widetilde{\tau}(\chi\circ N_{K/\Q})$ differ by $[N_{K/\Q}(\sqrt{-d})\mathrm{Norm}(p^{n}\roi_K)]^j = d^jp^{2nj}$ (see e.g.\ \cite[\S2.6]{Wil17}). Using the standard identity $\tau_K(\chi\circ N_{K/\Q})\tau_K(\chi^{-1}\circ N_{K/\Q}) = N_{K/\Q}(\mathrm{cond}(\chi\circ N_{K/\Q}))$ we move the Gauss sum to the denominator, as in Thm.\ \ref{thm:Q}.

If $p\roi_K = \pri^2$ is ramified (i.e. $t > 0$), then $\ff\defeq \mathrm{cond}(\chi \circ N_{K/\Q}) = \pri^{2n-t}$ need not be principal, so the computation is more involved. Noting $(d')^{j}p^{2nj} = d^jp^{(2n-t)j}$, one can show this by tracking through the definitions of $\varphi(x_{\ff}), \varphi_{\ff}(x_{\ff})$ and $\widetilde{\tau}(\varphi^{-1})$ in \cite{Wil17}. It is more convenient, though, to use the adelic formulation of the interpolation formula from \cite[Thm.\ 12.1]{BW_CJM}. From \cite[Def.\ 2.6]{BW_CJM}, we see $\tau_K(\varphi) = d^j\tau_K(\chi\circ N_{K/Q}) = d^{j}p^{2n-t}/\tau_K(\chi^{-1}\circ N_{K/\Q})$. We conclude as $\pi_{\ff}^{\mathbf{j}}$ in \cite[Thm.\ 12.1]{BW_CJM} is $p^{(2n-t)j}$ here. 
\end{proof}
By \eqref{eq:complex periods} and the identity $\alpha_{p\roi_K}(\f) = \alpha_p(f)^2$, we see that
\begin{equation}\label{eq:factorisation minus gauss}
\frac{(d')^{j+1}p^{2n(j+1)}\unitsize}{(-1)^k 2\alpha_{p\roi_K}(\f)^n\Omega_{\f}}\cdot  \tau(\chi_{K/\Q})^{-1} = \frac{(p^n)^{j+1}}{\alpha_p(f)^n\Omega_f^\pm}\cdot \frac{(d'p^n)^{j+1}}{\alpha_p(f)^n\Omega_f^\mp}.
\end{equation}
When $h < (k+1)/2$, Thm.\ \ref{thm:factorisation} now follows by combining \eqref{eq:factorisation minus gauss} with the identity
\[\tau_K(\chi\circ N_{K/\Q})\tau(\chi_{K/\Q}) = \tau(\chi)\tau(\chi\chi_{K/\Q})\]
of Gauss sums (applied to $\chi^{-1}$, noting $\chi_{K/\Q}^{-1} = \chi_{K/\Q}$). This is a characteristic $0$ version of the classical Hasse--Davenport identity. It can be verified locally by decomposing the Gauss sums into a product of local epsilon factors, as in \cite[Prop.\ 6.14]{Nar04}; the local result is proved, for example, in \cite[\S6]{Mar72}. This completes the proof in the slope $< (k+1)/2$ case.

\subsection{Proof of Thm.\ \ref{thm:factorisation} in general}
Now suppose $f$ has slope $h \geq \frac{k+1}{2}$. Both $L_p(f)L_p^{\chi_{K/\Q}}(f)$ and $L_p^{\cyc}(\f)$ are admissible of order $2h \geq k+1$, so are not determined by interpolation at critical values. To circumvent this, we use the three variable $p$-adic $L$-function through $\f$. 

Let $V_{\Q}$ be a neighbourhood of $x_f$ in $\CC$ lying over $\Sigma \subset \W_{\Q} \cong \W_{K,\mathrm{par}}$. Let $V$ denote the image of $V_{\Q}$ under the $p$-adic base-change map, a neighbourhood of $x_{\f}= \mathrm{BC}(x_f)$ in $\E_{\mathrm{bc}}$. For any classical point $y \in V_{\Q}$, write $f_y$ for the corresponding modular form, and write $\f_{y}$ for its base-change to $K$. 

The slope of a Coleman family is locally constant. We can shrink $V_{\Q}$ and $\Sigma$ so that: (1) along $V$, the slope at $p$ is constant, equal to $v_p(\alpha_{p\roi_K}(\f)) = 2v_p(\alpha_p(f)) = 2h$; and (2) any classical weight $\ell \in \Sigma\backslash\{k\}$ satisfies $\ell +1 > 2(k+1) \geq 2h$. If $y$ is a classical point in $V_{\Q}$ above such a weight $\ell$, then $v_p(\alpha_p(f_y)) = h < \frac{\ell+1}{2}$, so 
\[
L_p^\cyc(\f_y) = L_p(f_y)L_p^{\chi_{K/\Q}}(f_{y})
\]
by \S\ref{sec:very small slope}, where again we normalise the periods appropriately.

As $\f$ is $\Sigma$-smooth, by \S\ref{sec:subsec three variable} we have a three-variable $p$-adic $L$-function $\mathcal{L}_p(V)$ over $V$. We can restrict this to a two-variable function $\mathcal{L}^\cyc_p(V) \in \Dla(\Zp^\times, \cO(V))$.   
Composing with $\mathrm{BC}^* : \roi(V) \rightarrow \roi(V_{\Q})$, we view this in $\Dla(\Zp^\times,\cO(V_{\Q}))$.

\begin{proposition}\label{prop:factor family}
Suppose that $L_p^\cyc(\f)$ and $L_p(f)L_p^{\chi_{K/\Q}}(f)$ are non-zero. After possibly shrinking $V_{\Q}$ and rescaling by $\roi(V_{\Q})^\times$, we have a factorisation
\[
\mathcal{L}_p^\cyc(V) = \mathcal{L}_p(V_{\Q})\mathcal{L}_p^{\chi_{K/\Q}}(V_{\Q}).
\]
In particular, in $\Dla(\Zp^\times,\roi(V_{\Q}))$ we have an equality of $\roi(V_{\Q})$-lines 
\[
	\roi(V_{\Q})\cdot \mathcal{L}_p^\cyc(V) \ = \  \roi(V_{\Q}) \cdot \mathcal{L}_p(V_{\Q}) \mathcal{L}_p^{\chi_{K/\Q}}(V_{\Q}).
\]
\end{proposition}
In the general case, Thm.\ \ref{thm:factorisation} follows by specialising this identity at $x_f$.

Rescaling by $\roi(V_{\Q})$ corresponds to renormalising the complex and $p$-adic periods. For each classical $y \in V_{\Q}(L)$, we take $\Omega_{\f_y} = (-1)^\ell\unitsize\Omega_{f_y}^+\Omega_{f_y}^-/2\tau(\chi_{K/\Q}) \in \C^\times$. The renormalisation of $p$-adic periods is handled in the proof.

\begin{proof}
After taking the Amice transform, we may consider the functions in question as analytic functions on the two-dimensional rigid space $V_{\Q}\times\mathscr{X}(\Zp^\times)$, where, as in the introduction, we write $\mathscr{X}(\Zp^\times)$ for the rigid character space of $\Zp^\times$. Consider, then, the quotient
\[
 C = C(z,-) \defeq \frac{\mathcal{L}_p^\cyc(V)}{\mathcal{L}_p(V_{\Q})\mathcal{L}_p^{\chi_{K/\Q}}(V_{\Q})} \in \mathrm{Frac}\big(\roi(V_{\Q}\times\mathscr{X}(\Zp^\times))\big).
\]
This is well-defined by a similar argument to that in Prop.\ \ref{prop:well-defined}.
At each classical point $y \neq x_f$ in $V_{\Q}$, we have $C(y,\phi)$ is constant in $\phi$ using the factorisation at very small slope points. As such points are Zariski-dense, we deduce that $C(z,\phi)$ is constant in $\phi$ for any $z \in V_{\Q}$, so $C \in \mathrm{Frac}(\roi(V_{\Q}))$. Since (by assumption) neither $L_p^\cyc(\f)$ nor $L_p(f)L_p^{\chi_{K.\Q}}(f)$ is zero, $C$ does not have a zero or pole at $x_{f}$. Hence we may shrink $V_{\Q}$ further so that $C \in \roi(V_{\Q})^\times$. But this corresponds to renormalising the $p$-adic periods and completes the proof.
\end{proof}

\begin{remark}
By their constructions, $\mathcal{L}_p^\cyc(V)$ and $\mathcal{L}_p(V_{\Q})\mathcal{L}_p^{\chi_{K/\Q}}(V_{\Q})$ are themselves only well-defined up to multiplication by elements of $\roi(V_{\Q})^\times$, so this indeterminacy is expected. The non-vanishing condition is always satisfied if $f$ and $\f$ are non-critical by non-vanishing of a classical critical $L$-value.  As described in the introduction, it is conjectured that $L_p(f)$ and $L_p^{\chi_{K/\Q}}(f)$ are never zero, and similarly we conjecture that $L_p(\f)$ is never zero.
\end{remark}

\begin{remark}
Suppose $p$ is split and $f$ has level $N$ prime to $p$. Let $\alpha_p$ and $\beta_p$ denote the roots of the Hecke polynomial of $f$ at $p$, and assume $\alpha_p \neq \beta_p$. There are two possible $p$-stabilisations $f_\alpha$, $f_\beta$ of $f$ to level $pN$. The base-change $\f$ has \emph{four} possible $p$-stabilisations to level $pN\roi_K$; as in Rem.\ \ref{rem:missing cases}, we can consider $\f_{\alpha\alpha},\f_{\alpha\beta},\f_{\beta\alpha}$ and $\f_{\beta\beta}$. Then $\f_{\alpha\alpha}$ and $\f_{\beta\beta}$ are the base-changes of $f_{\alpha}$ and $f_{\beta}$, but $\f_{\alpha\beta}$ and $\f_{\beta\alpha}$ cannot be base-change themselves, as they have distinct eigenvalues at $\pri$ and $\pribar$. In this case, Loeffler and Zerbes have mentioned to the authors that $L_p^\cyc(\f_{\alpha\beta},\phi)$ can be expressed as a linear combination of the two products $L_p(f_\alpha,\phi)L_p^{\chi\FQ}(f_\beta,\phi)$ and $L_p(f_\beta,\phi)L_p^{\chi\FQ}(f_\alpha,\phi)$.
\end{remark}

\subsection{Non-criticality under base-change}
Let $f \in S_{k+2}(\Gamma_1(N))$ be a decent eigenform satisfying Conditions \ref{running assumptions Q}. If $f$ has non-critical slope then its base-change $\f$ can still have critical slope. However, critical slope forms can still be non-critical, so it is natural to ask: if $f$ is non-critical (resp.\ critical), is $\f$ non-critical (resp.\ critical)? Conjecturally, we can use Thm.\ \ref{thm:factorisation} to answer this positively:

\begin{corollary}\label{cor:non-critical2}
	\begin{itemize}\setlength{\itemsep}{0pt}
 \item[(i)] If $f$ is critical, then either: (i-a) $\f$ is critical; or (i-b) $\f$ is non-critical and $L_p(f)L_p^{\chi_{K/\Q}}(f) = 0$.
 \item[(ii)]  If $f$ is non-critical, then either: (ii-a) $\f$ is non-critical; or
 (ii-b) the weight map $\Epar \to \W_{K,\mathrm{par}}$ is \'etale at $x_{\f}$, and $L_p^{\cyc}(\f) = 0$; or
(ii-c) there is a classical family of non-base-change Bianchi cusp forms through $\f$.
\end{itemize}
\end{corollary}
Conjecturally $p$-adic $L$-functions are non-zero, and we expect eigenvarieties should not be \'etale at critical points, so (i-b) should not happen and (ii-b) should be doubly impossible; and (ii-c) contradicts Conj.\  \ref{conj:cal-maz precise}. Thus conjecturally $f$ is non-critical if and only if $\f$ is non-critical.
\begin{proof}
	Note from \cite[Thm.\ 2]{Bel12} that if $f$ is critical, then $L_p(f)$ vanishes at every critical $\varphi_{p-\mathrm{fin}}$. Similarly, if $\f$ is critical $\Sigma$-smooth, then $L_p(\f)$ vanishes at every critical $\varphi_{p-\mathrm{fin}}$ (Thm.\ \ref{prop:critical}). Thus if we have non-vanishing of a critical \emph{$p$-adic} $L$-value of $f$ (resp.\ $\f$), then $f$ (resp.\ $\f$) is necessarily non-critical.
	
	(i) Suppose (i-a) fails, i.e.\ $f$ is critical but $\f$ is non-critical, so $L_p(\f)$ exists by Thm.\ \ref{thm:Wil17}. There is a Dirichlet character $\varphi$ of conductor $p^n > 1$ such that
	\[ L(\f,\varphi\circ N_{K/\Q}, k+1) = L(f,\varphi,k+1)L(f,\varphi\chi\FQ,k+1) \neq 0.\]
	Indeed, if $k>0$, then for any $\varphi$, the Euler product expressions for $L(f,\varphi,k+1)$ and $L(f,\varphi\chi_{K/\Q},k+1)$ absolutely converge to non-zero complex numbers. If $k=0$, then this is a consequence of the main result of \cite{Roh84}.
	
	Since $\varphi$ has non-trivial $p$-power conductor, the $p$-adic $L$-functions $L_p(f)$, $L_p^{\chi\FQ}(f)$ and $L_p^{\cyc}(\f)$ do not have exceptional zeros at $\phi  = (\varphi|\cdot|^k)_{p-\mathrm{fin}}$; so as $\f$ is non-critical, by \eqref{eq:non-critical interpolation} we have $L_p^\cyc(\f) \neq 0$. If $L_p(f)L_p^{\chi_{K/\Q}}(f) \neq 0$, then by Thm.\ \ref{thm:factorisation}, up to non-zero rescaling we have $L_p^\cyc(\f) = L_p(f)L_p^{\chi_{K/\Q}}(f)$. We then have
	\begin{equation}\label{eq:non-vanishing}
	0\neq L_p(\f, \phi\circ N_{K/\Q}) = L_p(f,\phi)L_p^{\chi\FQ}(f,\phi),
	\end{equation}
	 meaning $f$ is non-critical, a contradiction. So $L_p(f)L_p^{\chi_{K/\Q}}(f) = 0$, and (i-b) holds.
	
	(ii) Suppose $f$ is non-critical. If $\f$ is non-critical, then (ii-a) occurs; and if $\f$ is critical and $\f$ is not $\Sigma$-smooth, then (ii-c) occurs by Props.\ \ref{prop:parallel weight} and \ref{prop:smooth nc}. So suppose $\f$ is critical and $\Sigma$-smooth. By Prop.\ \ref{prop:smooth nc} and $\Sigma$-smoothness, since $f$ is non-critical, $\mathrm{BC} : \CC \to \Epar$ is locally an isomorphism over $\W_{\Q} \cong \W_{K,\mathrm{par}}$ at $x_f$. By non-criticality, the weight map $\CC\to \W_{\Q}$ is \'etale at $x_f$ (e.g. \cite{Bel12}), so we see $\Epar \to \W_{K,\mathrm{par}}$ is \'etale at $x_{\f}$.
	
	Now suppose $L_p^\cyc(\f) \neq 0$. We know $L_p(f)L_p^{\chi_{K/\Q}}(f) \neq 0$ by non-criticality of $f$, so by Thm.\ \ref{thm:factorisation} up to non-zero rescaling we have $L_p^{\cyc}(\f) = L_p(f)L_p^{\chi_{K/\Q}}(f)$. By the same arguments to \eqref{eq:non-vanishing}, one can choose a locally analytic $\phi$ on $\Zp^\times$ such that
	\[
		L_p(\f, \phi\circ N_{K/\Q}) = L_p^{\cyc}(\f,\phi) = L_p(f, \phi)L_p^{\chi_{K/\Q}}(f,\phi) \neq 0,
	\]
	contradicting Thm.\ \ref{prop:critical} as $\f$ is critical. Thus $L_p^{\cyc}(\f) = 0$, and (ii-b) holds.
\end{proof}

\subsection{Restriction to the anticyclotomic line}
The methods of this section apply in another related case, the details of which we leave to the interested reader; we thank Lennart Gehrmann for pointing this out to us. By class field theory $\cl_K(p^\infty) \cong \mathrm{Gal}(K_\infty/K),$ where $K_\infty$ is the maximal abelian extension of $K$ unramified outside $p$. Above, we restricted to the cyclotomic subextension in $K_\infty$; we can also naturally restrict to the \emph{anticyclotomic} subextension $K_\infty^{\mathrm{anti}}/K$. The \emph{anticyclotomic $p$-adic $L$-function of $f$ over $K$} is a distribution $L_p^{\mathrm{anti}}(f)$ on $\mathrm{Gal}(K_\infty^{\mathrm{anti}}/K)$, admissible of order $h$, that satisfies the interpolation property that at a critical anticyclotomic character $\chi$ of $K$, we have
\[
 \big(L_p^{\mathrm{anti}}(f,\chi)\big)^2 = (*) \Lambda(\f,\chi),
\]
for an explicit factor $(*)$. These objects were introduced by Bertolini and Darmon in \cite{BD96} for ordinary elliptic curves, and general constructions now exist (e.g.\ \cite{Kim17}). If $h < \frac{k+1}{2}$, this interpolation property is enough to show that (after normalising the periods) we have 
\begin{equation}\label{anti-fact}
L_p^{\mathrm{anti}}(f)^2 = L_p^{\mathrm{anti}}(\f),
\end{equation}
where $L_p^{\mathrm{anti}}(\f)$ is the restriction of $L_p(\f)$ to the anticyclotomic line (see e.g.\ \cite{Geh17} for this result in the ordinary case). Suppose there exists such a two-variable function $\mathcal{L}_p^{\mathrm{anti}}(V_{\Q})$, over a neighbourhood $V_{\Q}$ in $\CC$, interpolating the anticyclotomic $p$-adic $L$-functions at classical weights. Then the methods of this section show that, under an analogous non-vanishing condition, and up to multiplication by an element of $\roi(V_{\Q})^\times$, we have an equality of two-variable distributions
\begin{equation}\label{eq:anticyc factorisation}
 \mathcal{L}_p^{\mathrm{anti}}(V_{\Q})^2 = \mathcal{L}_p^{\mathrm{anti}}(V).
\end{equation}
If $h \geq \frac{k+1}{2}$, we can obtain the identity (\ref{anti-fact}) for $\f$ by specialising \eqref{eq:anticyc factorisation} at $\f$.




\lhead{\emph{Appendix: A base-change deformation functor}}
\rhead{\emph{Wang-Erickson}}

\appendix
\addcontentsline{toc}{section}{Appendix: A base change deformation functor}
\stepcounter{section}
  \vspace*{10pt}
   \begin{center}
      \large\textbf{Appendix: A base-change deformation functor}\\[8pt]
      \large{by Carl Wang-Erickson\footnote{C.W.E.\ was supported by EPSRC grant EP/L025485/1. 
}}
   \end{center}
   \vspace*{10pt}

The point of this appendix is to supply the proof of Prop.\ \ref{prop:appendix}, regarding deformations of Galois representations. The main idea we will apply here applies under the following running assumptions: 
\begin{enumerate}[label=(\Alph*)]\setlength{\itemsep}{1pt}
\item there is an index 2 subgroup $H \subset G$ and a chosen element $c \in G \smallsetminus H$ of order 2. Equivalently, $G$ is expressed as a semi-direct product $H \rtimes \langle c\rangle$;
\item $\mathrm{char}(L) \neq 2$, for $L$ the base coefficient field of the deformed representation. 
\end{enumerate}

In the first section we set up the theory of the base change deformation functor. In the second section, we verify that this theory is compatible with arithmetic conditions imposed when $G$ is a Galois group over $\Q$.

\subsection{The base change deformation functor} 
\label{subsec: BC def}
We work under assumptions (A)-(B) above. Let $\rho : G \to \GL_d(L)$ be a representation that is absolutely irreducible after restriction to $H$. Let $\cA_L$ be the category of Artinian local $L$-algebras $(A, \m_A)$ with residue field $L$. We denote by $X$ the deformation functor for $\rho\vert_H$. This is the functor from $\cA_L$ to the category of sets given by
\begin{equation}
\label{eq: X functor}
A \mapsto \{\tilde \rho_A : H \to \GL_d(A) \mid (\tilde \rho_A \mod{\m_A}) = \rho\vert_H\}/\sim,
\end{equation}
where $\sim$ is the equivalence relation of ``strict equivalence,'' that is, conjugation by $1 + M_d(\m_A) \subset \GL_d(A)$. We will let $\rho_A \in X(A)$ denote a \emph{deformation} of $\rho\vert_H$ with coefficients in $A$. This is in contrast to the notation $\tilde\rho_A$, which we reserve for a \emph{lift} of $\rho\vert_H$ to $A$, i.e.\ a homomorphism $\tilde\rho_A \in \rho_A$ as in \eqref{eq: X functor}. 

Let $X^\mathrm{bc}$ denote the subfunctor of $X$ cut out by the condition that some (equivalently, all) $\tilde\rho_A \in \rho_A$ admits an extension to a homomorphism $\tilde\rho^G_A : G \to \GL_d(A)$ such that $\tilde\rho^G_A\vert_H = \tilde\rho_A$. In this case, we say that $\rho_A$ admits an extension to an $A$-valued deformation $\rho^G_A$ of $\rho$.

For $h \in H$, we write $h^c := chc \in H$ for twisting by $c$. Likewise, for a group homomorphism $\eta$ with domain $H$, let $\eta^c(h) := \eta(h^c)$. 

\begin{lemma}
\label{lem: twist extension}
Let $A \in \cA_L$ and $\rho_A \in X(L)$. Then $\rho_A$ admits an extension to $G$ deforming $\rho$ if and only if there exists $\tilde\rho_A \in \rho_A$ such that
\begin{equation}
\label{eq: fixed point}
\ad\rho(c) \cdot \tilde \rho_A^c =  \tilde\rho_A.
\end{equation}
\end{lemma}

\begin{proof}
Assume that there exists $\tilde\rho_A \in \rho_A$ and $\tilde\rho^G_A : G \to \GL_d(A)$ such that $\tilde\rho^G_A\vert_H = \tilde\rho_A$. Because the characteristic of $L$ is not 2, the deformation functor for $\rho\vert_{\langle c\rangle}$ is trivial; compare the proof of \cite[Prop.\ 5.3.2]{CWE2018}. Equivalently, there exists some $x \in 1 + M_d(\m_A) \subset \GL_d(A)$ such that $\ad x \cdot \tilde\rho^G_A(c) = \rho(c)$. Then one readily observes that $\ad x \cdot \tilde\rho_A$ is a solution to \eqref{eq: fixed point}. 

Next we prove the converse. Assume that we have $\tilde \rho_A$ solving \eqref{eq: fixed point}. Then we define $\tilde\rho^G_A : G \to \GL_d(A)$ by
\[
\tilde\rho^G_A(g) := \left\{
\begin{array}{ll}
\tilde\rho_A(g) & \text{for } g \in H, \\
\rho(c)\tilde\rho_A(h) & \text{for } g = ch,  h \in H.
\end{array}\right.
\]
It is then straightforward to calculate that $\tilde\rho^G_A$ is a group homomorphism such that $\tilde\rho^G_A\vert_H = \tilde\rho_A$.
\end{proof}

Notice that the map of lifts $\tilde\rho_A$ of $\rho\vert_H$ to $A$ sending
\[
\tilde \iota: \tilde\rho_A \mapsto \ad\rho(c) \cdot \tilde\rho_A^c
\]
is an involution on lifts of $\rho\vert_H$. Its fixed points are exactly those lifts satisfying \eqref{eq: fixed point}. This involution descends to an functorial involution of deformations
\[
\iota : X(A) \to X(A).
\]
To justify this claim, we calculate that for any $x \in \GL_d(A)$, 
\[
\tilde\iota(\ad x \cdot \tilde\rho_A) = \ad\rho(c) \cdot \ad x \cdot \tilde\rho_A^c = \ad y \cdot (\tilde \iota (\tilde\rho_A)),
\]
where $y = \ad\rho(c) \cdot x$. 

Let $X^\iota$ denote the $\iota$-fixed subfunctor of $X$, and let $\mathfrak{t}$ (resp.\ $\mathfrak{t}^\mathrm{bc}$) denote the tangent space $X(L[\varepsilon]/(\varepsilon^2))$ (resp.\ $X^\mathrm{bc}(L[\varepsilon]/(\varepsilon^2))$.

 \begin{proposition}
 \label{prop: bc main}
 \begin{itemize}
 \item[(i)] There is a canonical isomorphism $X^\iota \cong X^\mathrm{bc}$.
 \item[(ii)] The deformation problems $X^\mathrm{bc}, X$ on $\mathcal{A}_L$ are pro-represented by pro-objects $R^\mathrm{bc}, R \in \mathcal{\hat A}_L$. The involution $\iota$ induces an automorphism $\iota^* : R \to R$, and there is a natural surjection
 \[
  R \twoheadrightarrow  R^\mathrm{bc} := \textstyle\frac{R}{((1 - \iota^*)(R))}.
 \]
 \item[(iii)] There is a canonical injection $\mathfrak{t}^\mathrm{bc} \hookrightarrow \mathfrak{t}$ of tangent spaces. The image of this injection is the subspace $\mathfrak{t}^\iota \subset \mathfrak{t}$ fixed by the involution $\iota_* : \mathfrak{t} \to \mathfrak{t}$ induced by $\iota$.
 \end{itemize}
 \end{proposition}
 
 \begin{proof}
Part (i) follows directly from Lem.\ \ref{lem: twist extension}.
 
For Part (ii), it is well-known that $X$ is pro-representable; see e.g.\ \cite{Maz89}. It is a brief exercise that a homomorphism $R \to A$ kills $(1-\iota^*)(R)$ if and only if the corresponding deformation of $\rho\vert_H$ is $\iota$-fixed. Then the pro-representability of $X^\mathrm{bc}$ by $R^\mathrm{bc}$ follows from (i).
 
Part (iii) follows from Part (ii) and the perfect $L$-linear duality of $\m_R/\m_R^2$ and $X(L[\varepsilon]/(\varepsilon^2))$. 
 \end{proof}

\subsection{Galois-theoretic conditions}

Let $G = G_{\Q,S}$ and $H = G_{K,S_K}$ (see Not.\ \ref{defn:restricted ramification}). We also use the decomposition groups and complex conjugation $c \in G$ given in \eqref{eq: decomp groups}. 
The data $(G,H,c)$ satisfy assumption (A), as $K/\Q$ is imaginary quadratic. 
 
Because the level of the modular form $f$ of Prop.\ \ref{prop:smooth} is supported by $S$, and because $p, \infty \in S$, the representation $\rho_f$ of the absolute Galois group of $\Q$ factors through $G_{\Q,S}$. We let $\rho \defeq \rho_f : G \to \GL_2(L)$, as in Def.\ \ref{def:decent}, with its critical refinement with eigenvalue $\alpha_p$. It is an $L$-linear representation, where $L$ is a $p$-adic field; thus we have satisfied assumption (B).
 
Deformation theory as in \S\ref{subsec: BC def} can be carried out for \emph{continuous} representations of $G$ and $H$, using the $p$-adic topology of $L$, and the arguments therein make good sense in this setting. This is standard; see e.g.\ \cite[\S9]{Kisin2003}. From now on, we impose continuity without further comment. 

Because $G$ and $H$ satisfy the finiteness condition $\Phi_p$ of \cite[\S1.1]{Maz89}, it follows that the deformation rings $R, R^\mathrm{bc}$ of Prop.\ \ref{prop: bc main} representing $X, X^\mathrm{bc}$ are Noetherian and (equivalently) $\mathfrak{t}, \mathfrak{t}^\mathrm{bc}$ have finite $L$-dimension. 

\begin{lemma}
Conditions (i) and (ii) of Def.\ \ref{def:deformation problem proof critical case} determine a subfunctor $X^\mathrm{ref} \subset X$ that is Zariski-closed, hence representable by a quotient ring $R \twoheadrightarrow R^\mathrm{ref}$. 
\end{lemma}
\begin{proof}
This is standard -- see e.g.\ \cite[p.26]{Berg17} and \cite[Prop.\ 8.13]{Kisin2003}. In particular, the important assumption \cite[(8.8.1)]{Kisin2003} is satisfied because $f$ has been critically refined. 
\end{proof}

\begin{proof}[Proof of Prop.\ \ref{prop:appendix}]
Because both the ``ref'' and ``bc'' conditions have been shown to be Zariski-closed conditions on $X$, their intersection functor $X^{\mathrm{ref}, \mathrm{bc}}$ is representable by a quotient $R^\mathrm{ref} \twoheadrightarrow R^{\mathrm{ref}, \mathrm{bc}}$. Then apply Prop.\ \ref{prop: bc main} and its proof. 
\end{proof}

To make Prop.\ \ref{prop:appendix} useful, we check that the properties of a $G$-deformation $\rho_A^G$ of $\rho_f$ guaranteeing that $\rho_A^G\vert_H$ determines a point of $X^\mathrm{ref}$ (and, consequently, a point of $X^{\mathrm{ref}, \mathrm{bc}}$) are what we would naturally expect them to be. 

\begin{lemma}
\label{lem: Q to K}
Let $\rho_A^G$ be a deformation of $\rho_f : G \to \GL_2(L)$ to $A \in \cA_L$. Then $\rho_A^G\vert_H \in X^\mathrm{ref}(A)$ if and only if $\rho_A^G$ satisfies
\begin{enumerate}[label=(\roman*)]\setlength{\itemsep}{1pt}
\item For primes $q \mid N$ such that $q \neq p$, $\rho_A^G\vert_{I_q} \simeq \rho\vert_{I_q} \otimes_L A$. 
\item The restriction $\rho_A^G\vert_{G_p}$ has 
\begin{enumerate}[label=(\arabic*)] \setlength{\itemsep}{1pt}
\item one Hodge--Sen--Tate weight is constant and equal to 0, and 
\item there exists $\widetilde{\alpha}_p \in A$ such that the $A$-module $D_\mathrm{crys}(\rho_A^G\vert_{G_p})^{\varphi = \widetilde{\alpha}_p}$ is free of rank $1$ and $(\widetilde{\alpha}_p \mod{\m_A}) = \alpha_p$.
\end{enumerate}
\end{enumerate}
\end{lemma}

\begin{proof}
It is a straightforward exercise about representations and the corresponding Frobenius isocrystals to verify that the statements of (i)-(ii) of Lem.\ \ref{lem: Q to K} are equivalent to (i)-(ii) of Def.\ \ref{def:deformation problem proof critical case} under both extension and restriction.
\end{proof}

\scriptsize
\renewcommand{\refname}{\normalsize References} 
\bibliography{master_references}{}
\bibliographystyle{alpha}

 \lhead{\emph{Families of Bianchi $p$-adic $L$-functions}}
\rhead{\emph{Barrera Salazar and Williams}}
\Addresses

\end{document}